\documentclass[11pt, oneside]{amsart}
\usepackage{amsmath,ifthen, amsfonts, amssymb,
srcltx, amsopn, textcomp}

\usepackage{hyperref}
\usepackage{graphicx, overpic}

\usepackage[small]{caption}

\newcommand{\showcomments}{yes}
 \renewcommand{\showcomments}{no}

\newsavebox{\commentbox}
\newenvironment{com}%
{\ifthenelse{\equal{\showcomments}{yes}}%
{\footnotemark
        \begin{lrbox}{\commentbox}
        \begin{minipage}[t]{1.25in}\raggedright\sffamily\tiny
        \footnotemark[\arabic{footnote}]}
{\begin{lrbox}{\commentbox}}}%
{\ifthenelse{\equal{\showcomments}{yes}}%
{\end{minipage}\end{lrbox}\marginpar{\usebox{\commentbox}}}
{\end{lrbox}}}

\numberwithin{equation}{section}
\theoremstyle{plain}
\newtheorem{theorem}{Theorem}
\numberwithin{theorem}{section}
\newtheorem{corollary}[theorem]{Corollary}
\newtheorem{lemma}[theorem]{Lemma}

\newtheorem{proposition}[theorem]{Proposition}

\newtheorem{claim}[theorem]{Claim}

\newtheorem*{namedtheorem}{\theoremname}
\newcommand{\theoremname}{testing}
\newenvironment{named}[1]{\renewcommand{\theoremname}{#1}\begin{namedtheorem}}{\end{namedtheorem}}

\theoremstyle{definition}

\newtheorem{remark}[theorem]{Remark}
\newtheorem{example}[theorem]{Example}

\newtheorem{construction}[theorem]{Construction}
\newtheorem{prob}[theorem]{Problem}

\newcommand{\refthm}[1]{Theorem~\ref{Thm:#1}}
\newcommand{\reflem}[1]{Lemma~\ref{Lem:#1}}
\newcommand{\refprop}[1]{Proposition~\ref{Prop:#1}}
\newcommand{\refcor}[1]{Corollary~\ref{Cor:#1}}
\newcommand{\refrem}[1]{Remark~\ref{Rem:#1}}
\newcommand{\refex}[1]{Example~\ref{Ex:#1}}
\newcommand{\refclaim}[1]{Claim~\ref{Claim:#1}}

\newcommand{\refeqn}[1]{\eqref{Eqn:#1}}
\newcommand{\refitm}[1]{\eqref{Itm:#1}}

\newcommand{\refsec}[1]{Section~\ref{Sec:#1}}

\DeclareMathOperator{\stab}{Stab}

\newcommand{\dist}{\textup{\textsf{d}}}

\newcommand{\field}[1]{\mathbb{#1}}
\newcommand{\integers}{\ensuremath{\field{Z}}}

\newcommand{\naturals}{\ensuremath{\field{N}}}
\newcommand{\reals}{\ensuremath{\field{R}}}
\newcommand{\complexes}{\ensuremath{\field{C}}}

\newcommand{\hyperbolic}{\ensuremath{\field{H}}}

\newcommand{\RR}{\reals}

\newcommand{\ZZ}{\integers}
\newcommand{\HH}{\hyperbolic}
\newcommand{\CC}{\complexes}

\newcommand{\boundary}   {{\ensuremath \partial}}
\newcommand{\bdy}{\boundary}

\newcommand{\nclose}[1]{\ensuremath{\langle\!\langle#1\rangle\!\rangle}}

\newcommand{\calX}{\mathcal X}
\newcommand{\calY}{\mathcal Y}
\newcommand{\calC}{\mathcal C}

\DeclareMathOperator{\girth}{\text{girth}}

\begin{document}

\title{Growth of quasiconvex subgroups}

\author{Fran\c{c}ois Dahmani}
\address{Universit\'e  Grenoble Alpes \\
Institut Fourier (UMR 5582) \\
100 rue des maths,   CS 40700. \\
F- 38 058 Grenoble, cedex 9 \\
France}
\email{francois.dahmani@ujf-grenoble.fr}
\author{David Futer}
\address{Dept. of Mathematics \\ Temple University \\ Philadelphia, PA 19122 \\ USA}
\email{dfuter@temple.edu}
\author{Daniel T. Wise}
           \address{Dept. of Math. \& Stats.\\
                    McGill University \\
                    Montreal, QC, Canada H3A 0B9 }
           \email{wise@math.mcgill.ca}
\date{\today}
\thanks{Dahmani was supported by the IUF.
  Futer was supported by NSF grant DMS--1408682 and the Elinor Lunder 
  Membership at the Institute for Advanced Study.
Wise was supported by NSERC}

\maketitle

\begin{abstract}
We prove that non-elementary hyperbolic groups grow exponentially more quickly than their
infinite index quasiconvex subgroups. The proof uses the classical tools of automatic structures and Perron--Frobenius theory.

We also extend the main result to relatively hyperbolic groups and cubulated groups. These extensions use the notion of growth tightness and the work of Dahmani, Guirardel, and Osin on rotating families.
\end{abstract}

\section{Introduction}

Consider a group $G$ acting by isometries on a graph $\Upsilon$, 
properly and cocompactly. 
One important special case is $\Upsilon(G,S)$, namely the Cayley graph of $G$ with respect to some finite generating set $S$, but $\Upsilon$ can also be more general. Fix a basepoint $b \in \Upsilon$ and a subset $H \subset G$.
The \emph{growth function} of $H$ is the function 
\begin{equation}\label{Eqn:fH}
f_{H,\Upsilon}(n) = \# \left\{ h\in H: \dist_\Upsilon(b, \, h(b))\leq n \right\} .
\end{equation}

Since $G$ is a quotient of a free group, and the action on $\Upsilon$ is proper, the growth function $f_{G,\Upsilon}$ is no larger than exponential. For any $H \subset G$, define the \emph{growth rate of $H$} to be
\begin{equation}\label{Eqn:LambdaDef}
\lambda_H = \lambda_H(\Upsilon) = \limsup_{n \to \infty} \sqrt[n]{ f_{H,\Upsilon}(n) }  .
\end{equation}
 We emphasize that the limit depends a great deal on $\Upsilon$. However, the triangle inequality in $\Upsilon$ implies that $\lambda_H$ is independent of $b$.

In this paper, we will be concerned with groups acting on hyperbolic metric spaces in the sense of Gromov, including groups that are themselves hyperbolic. We refer the reader to  \cite{Gromov87} and \cite{ABC91} for standard definitions regarding hyperbolic metric spaces and the groups that act on them.

Cannon showed that when $G$ is hyperbolic, the sequence $\sqrt[n]{ f_{G,\Upsilon}(n) }$  converges to $\lambda_G$ \cite{cannon:cocompact-hyp-groups, cannon:trieste-survey}; see \refcor{CannonGrowth} below.
 Coornaert showed the stronger result that $f_{G,\Upsilon}(n)$ is bounded above and below by constants times $(\lambda_G)^n$ \cite[Theorem 7.12]{Coornaert93}.
%
In this paper, we prove that there is a definite separation between the growth rate of $G$ and of any infinite index quasiconvex subgroup of $G$.

\begin{theorem}\label{Thm:GrowthGeneral}
Let $G$ be a non-elementary hyperbolic group acting properly and cocompactly on a graph $\Upsilon$.
Let $H$ be a quasiconvex subgroup of infinite index. Then
\[
\lambda_H(\Upsilon) < \lambda_G (\Upsilon).
\]
\end{theorem}

Recall that a group is \emph{elementary} if it contains a cyclic subgroup of finite index. The growth function of an elementary group $G$ is at most linear, hence $\lambda_G = 1$ for such a group. Since $\lambda_H \geq 1$ for any $H$, the ``non-elementary'' hypothesis is necessary.
The hypothesis that $H$ is quasiconvex is also necessary, as demonstrated in Examples~\ref{Ex:Fiber} and~\ref{Ex:Rips}.

\refthm{GrowthGeneral} has extensions to relatively hyperbolic groups and cubulated groups, as follows. We refer the reader to \refsec{BeyondHyperbolic} for the relevant definitions.

\begin{theorem}\label{Thm:RelHyperbolic}
Let $(G,\mathcal{P})$ be a non-elementary relatively hyperbolic group, and $H$ a relatively
quasiconvex subgroup of $(G,\mathcal{P})$ of infinite index in $G$. Suppose that $G$ acts properly and cocompactly on a graph $\Upsilon$.
Then
\[
\lambda_H(\Upsilon) < \lambda_G (\Upsilon).
\]
\end{theorem}

\begin{theorem}\label{Thm:Cubulated}
Let $G$ be a non-elementary group, acting properly and cocompactly on a CAT(0) cube complex $\calX$.
 Suppose that $\calX$ does not decompose as a product. Then, for every subgroup $H \subset G$ stabilizing an essential hyperplane of $\calX$, we have
\[
\lambda_H(\calX) < \lambda_G (\calX).
\]
\end{theorem}

Given these results, one may ask whether there exists a uniform upper bound $\alpha < \lambda_G$
such that $\lambda_H \leq \alpha$ for each infinite index quasiconvex subgroup $H \subset G$. In \refthm{NonUniform}, we construct an example (with $G$ a free group) showing that no such uniform bound can exist.

\subsection{Classical tools}
Cannon's study of the growth rate $\lambda_G$ came a consequence of his construction of automatic structures for hyperbolic groups.
The corresponding counts of rooted paths in a directed graph are closely tied to the theory of Perron--Frobenius eigenvalues of non-negative matrices. In keeping with this classical perspective, our proof  of \refthm{GrowthGeneral} primarily relies on automatic structures and Perron--Frobenius theory. We review the relevant material in
Sections~\ref{Sec:Automatic} and \ref{Sec:Perron}.

A crucial consequence of the classical theories, encapsulated in \refthm{PathCount}, is that
the growth rate $\lambda_G$ can be expressed as
the Perron--Frobenius eigenvalue of a transition matrix. Although this statement is surely known to experts, we were unable to find a sufficiently general statement in the literature. Thus we wrote down a proof in \refsec{GrowthRate}. We hope that a written account of \refthm{PathCount} will be useful to other researchers.

In the special case where $G$ is torsion-free and $\Upsilon$ is its Cayley graph with respect to some generating set, \refthm{GrowthGeneral} turns out to be a fairly quick consequence of  \refthm{PathCount} and the construction of a free product $H * \ZZ$ inside $G$.
To aid the reader, we treat this special case of \refthm{GrowthGeneral} in \refsec{TorsionFree}. See \refthm{GrowthTorsionFree}.

Proving the general case of \refthm{GrowthGeneral} requires certain generalizations of classical results about automatic structures (\refthm{Automatic}), growth rates (\refthm{PathCount}), and free products  (\refthm{FreeProduct}). In  \refsec{GeneralHyp}, we bootstrap from classical results to derive these general statements. Then we complete the proof of \refthm{GrowthGeneral}, following the same outline as \refthm{GrowthTorsionFree} while carrying some extra structure.

\subsection{Modern extensions}
All of the results and tools used in the proof of \refthm{GrowthGeneral} were known by 1990. By contrast, the proofs of Theorems~\ref{Thm:RelHyperbolic} and \ref{Thm:Cubulated} use a number of tools from modern geometric group theory. 

The first such tool is the notion of growth tightness, introduced by Grigorchuk and de la Harpe  \cite{grigorchuk-delaharpe}, which roughly states that a group $G$ grows faster than its quotients. See \refeqn{GrowthTight} for a precise definition. Arzhantseva and Lysenok  \cite{arzhantseva-lysenok} showed that hyperbolic groups are growth tight, by first proving a related result about regular languages for such a group (see \refthm{StrongContainmentGrowth} below). As we describe in \refsec{Alternate}, this point of view gives an alternate approach to \refthm{GrowthGeneral}. 

In \refsec{Alternate}, we also explain that  \refthm{GrowthGeneral} follows from a very recent theorem of Matsuzaki, Yabuki, and Jaerisch  on Patterson--Sullivan measures \cite{MYJ}.
 Their result is stated in very different language, and we explain the translation in \refsec{Measure}.

The second key tool for our approach is the idea of weakly properly discontinuous (WPD) elements acting on hyperbolic metric spaces, introduced by Bestvina and Fujiwara \cite{Bestvina-Fujiwara:bounded-cohomology}. Given such a $G$--action on $X$, and a loxodromic element $g \in G$, Dahmani, Guirardel, and Osin showed that there is some power $g^n$ whose entire normal closure $\nclose{g^n}$ acts loxodromically \cite{DGO}. As a consequence, any subgroup $H \subset G$ that acts elliptically on $X$ will survive in the quotient $G/ \nclose{g^n}$. When the $G$--action on a graph $\Upsilon$ is growth tight, this implies that $\lambda_H(\Upsilon) < \lambda_G(\Upsilon)$; we record this fact in \refprop{SmallGrowthTight}. Although the proof is very short, the result is quite general, and may be useful elsewhere.

Both Theorems~\ref{Thm:RelHyperbolic} and \ref{Thm:Cubulated} are proved in \refsec{BeyondHyperbolic}, by relying on \refprop{SmallGrowthTight}. For both cubulated groups and relatively hyperbolic groups, growth tightness is known by results of Arzhantseva, Cashen, and Tao \cite{arzhantseva-cashen-tao} and Yang \cite{Yang:GrowthTightness}. Thus the challenge is to find an appropriate action on a hyperbolic space $X$.
For relatively hyperbolic groups, we use the work of Hruska and Wise \cite{HruskaRelQC} and Dahmani and Mj \cite{Dahmani-Mj} to repeatedly cone off the Cayley graph of $G$ until we obtain an appropriate space $X$. For cubulated groups, the hyperbolic space $X$ is the contact graph $\calC \calX$ of a CAT(0) cube complex $\calX$, which has the right properties by a theorem of Behrstock, Hagen, and Sisto \cite{BHS:curve-complex-cubical}. Thus, in both cases, we conclude that $G$ grows faster than its subgroups.

In \refsec{Examples}, we prove \refthm{NonUniform}, which shows that quasiconvex subgroups of a free group $G$ can have growth rates approaching that of $G$ itself. We also present some examples and problems that are motivated by this work.

\subsection*{Acknowledgments} We thank Danny Calegari and Moon Duchin for helpful conversations.
We are also grateful to Goulnara Arzhantseva,  Chris Cashen, Ilya Gekhtman, Rita Gitik, Mark Hagen, Eduardo Mart\'{i}nez-Pedroza, Thomas Ng, and Sam Taylor for their helpful comments and corrections on an earlier draft of this paper.

\section{Automatic Group Background}\label{Sec:Automatic}

This section recalls some standard facts about automata and regular languages. We refer to \cite{Epstein92} for background and context. See also \cite[Section 3]{calegari:ergodic-groups} for a rapid survey.

Let $S$ be a finite set of letters, called an \emph{alphabet}. A \emph{word} in $S$ is a finite sequence of letters of $S$.
A \emph{language} $L$ over $S$ is a subset of the set of all words.
A \emph{finite state automaton} over $S$ is a finite directed graph $\Gamma$ whose edges are labeled by elements of $S$,
with one of its vertices declared to be a \emph{start state}, and some of its vertices declared to be \emph{accept states}.
Every automaton in this paper will be assumed to be \emph{deterministic}, meaning that no vertex has two outgoing edges with the same label.

 A directed path $e_1e_2\cdots e_r$ in $\Gamma$ \emph{reads} a word consisting of the sequence
$s_1s_2\cdots s_r$ of letters labeling its edges.
A word is \emph{accepted} by $\Gamma$ if it is read by a directed path from a start state to an accept state. A language is \emph{regular} if it is the full set of words accepted by some finite state automaton.

A vertex $v$ of an automaton $\Gamma$ is \emph{redundant} if it is not traversed by any path
from the start state to an accept state. Note that a redundant vertex can be removed without any effect on the language accepted by $\Gamma$. The automaton $\Gamma$ is \emph{pruned} if it contains no redundant vertices.

%

Let $G$ be a group with a finite, symmetric generating set $S$.
 An \emph{automatic structure} for $(G,S)$ is a regular language $L$ over $S$, along with positive constants $\kappa, \epsilon, \chi$, such that
\begin{enumerate}
\item Each word of $L$ is a $(\kappa,\epsilon)$-quasigeodesic in the Cayley graph $\Upsilon(G,S)$, starting at the identity.
\item  The map $L\rightarrow G$ is surjective.
\item If $w_1,w_2\in L$ and $s\in S$ satisfy $w_1s=_G w_2$,
then $w_1,w_2$ $\chi$-fellow travel in $\Upsilon(G,S)$.
\end{enumerate}
See \cite[Definition 11.26]{cannon:trieste-survey}. We refer to \cite{Epstein92} and \cite{gersten-short} for further details and other equivalent formulations.
We will not make use of the fellow traveling condition $(3)$.

Any total ordering of the alphabet $S$ induces a \emph{lexicographic ordering} on a language $L$ over $S$. Thus, although geodesics in a Cayley graph $\Upsilon(G,S)$ can be non-unique, for any $g \in G$ one may pick out the geodesic from $1$ to $g$ that is lexicographically first. The language of lexicographically-first geodesics is called a \emph{short-lex geodesic language} for $G$.

We may now state Cannon's celebrated result \cite[Theorem 11.27]{cannon:trieste-survey}. Compare \cite[Theorem 3.2.2]{calegari:ergodic-groups}.

\begin{theorem}[Hyperbolic Automatic]\label{Thm:Automatic}
Let $G$ be a hyperbolic group, with a  finite symmetric generating set $S$.
Choose an ordering on $S$, and let $L$ be the resulting short-lex geodesic language. Then $L$ is regular,
 maps bijectively to $G$, and endows $G$ with an automatic structure.
\end{theorem}

	We will also need the following result of Gersten and Short about quasiconvex subgroups \cite[Theorem~2.2]{gersten-short}.

\begin{theorem}[Quasiconvex Automatic]\label{Thm:QuasiconvexRational}
Let $G$ be a hyperbolic group with an automatic structure $L_G$. Let $H$ be a quasiconvex subgroup. Then the sub-language  $L_H\subset L_G$ of words that map to $H$ is itself a regular language.
\end{theorem}

In \refsec{ConstructAction}, we will generalize these theorems to the context of group actions on graphs.

\section{Perron--Frobenius Background}\label{Sec:Perron}

As with \refsec{Automatic}, the material in this section is mostly classical and standard. Excellent references include \cite{minc:nonnegative} and \cite[Section V.5]{flajolet-sedgewick}. See also \cite{calegari:ergodic-groups} for a brief summary from a group--theoretic perspective.

The main result of this section, \refthm{PathCount}, should be considered a folklore theorem. Although the result is very likely known to experts, we could not find a sufficiently general version in the literature. Therefore, we wrote down a proof.

\subsection{Irreducible Matrices}
For a matrix $B$, let $B_{ij}$ denote the $ij$-entry. A square matrix $A$ is \emph{irreducible} if for each $i,j$, there exists $n>0$ such that $(A^n)_{ij} > 0$.
By convention, we also regard the $1\times 1$ matrix $[0]$  as irreducible.
The following version of the Perron--Frobenius theorem summarizes the basic properties of non-negative irreducible matrices.

\begin{theorem}[Irreducible Perron--Frobenius]\label{Thm:PerronIrreducible}
Let $A$ be an irreducible, non-negative matrix that is not $[0]$. Let $\rho_A \geq 0$ be the largest absolute value of an eigenvalue of $A$. Then

\begin{enumerate}
\item $\rho_A > 0$, and is itself an eigenvalue.
\item If $A$ is integral, then $\rho_A\geq 1$.
\item\label{Itm:Eigenspace} Each of the left and right $\rho_A$--eigenspaces is
spanned by a single positive vector.
\item Any non-negative (left or right) eigenvector has eigenvalue $\rho_A$.
\item\label{Itm:Period} There are $h \geq 1$ eigenvalues of absolute value $\rho_A$. Furthermore, the spectrum of $A$ is invariant by a rotation of $\CC$ through angle  $2\pi/h$.
\end{enumerate}
\end{theorem}

The integer $h=h(A)$ is called the \emph{period} of the matrix $A$, and $\rho_A$ is called its \emph{Perron--Frobenius eigenvalue}. If $h(A) = 1$, we say $A$ is \emph{aperiodic}.

\begin{proof}
This is a combination of several results in  \cite{minc:nonnegative}. See Theorem I.4.1, Corollary I.4.2, Theorem I.4.4, and Theorem III.1.2.
\end{proof}

The discussion of irreducible matrices connects to the above discussion of regular languages as follows.

 Let $\Gamma$ be a directed graph with vertices  $v_1, \ldots, v_k$.
  We say  $v_j$ is \emph{reachable} from $v_i$ if there is a directed path from $v_i$ to $v_j$.
 The graph is called \emph{strongly connected} if every vertex is reachable from every other. Any finite directed graph decomposes into a collection \emph{strongly connected components}, linked by a directed acyclic graph. See \cite[Figure V.15]{flajolet-sedgewick}.

Now, suppose that $\Gamma$ is a finite state automaton
with start state $v_1$,
 and let $L$ be the regular language accepted by $\Gamma$.
We gain information about the growth of $L$ by studying the adjacency matrix $A = A(\Gamma)$, where the entry $A_{ij}$ equals the number of edges from $v_i$ to $v_j$. The following lemma is a nearly immediate consequence of the definition of $A$. Compare \cite[Theorems IV.3.2 and IV.3.3]{minc:nonnegative}.

\begin{lemma}\label{Lem:IJCount}
Let $\Gamma$ be a finite directed graph with adjacency matrix $A$.
Then
\begin{enumerate}
\item $(A^n)_{ij}$ is the number of  length~$n$ paths from $v_i$ to $v_j$.
\item $\Gamma$ is strongly connected if and only if $A$ is irreducible.
\item If $\Gamma$ is strongly connected and has at least one cycle, the period $h(A)$ is equal to the greatest common divisor of the lengths of cycles in $\Gamma$.
\end{enumerate}
\end{lemma}

\subsection{Weighted graphs}

For some of our applications, we need to work in the more general context of weighted graphs. A \emph{weighted graph} is a (finite) directed graph with positive real numbers (called \emph{weights}) assigned to edges.
Define the \emph{weight} of a path to be the product of the weights of its edges, counted with multiplicity.

For every pair of vertices $v_i, v_j$ of a weighted graph $\Gamma$, we shall assume without loss of generality that there is at most one edge from $v_i$ to $v_j$. This is because several weighted edges can be replaced by a single edge, labeled by the sum of the weights.

Weighted   graphs are in $1$--$1$ correspondence with non-negative square matrices. This can be seen as follows. Given a weighted graph $\Gamma$, the matrix $A(\Gamma)$ has entry $A_{ij}$ equal to the  weight on the edge from $v_i$ to $v_j$ (or $0$ if no such edge exists). Conversely, given a non-negative matrix $A$, we may recover a weighted graph $\Gamma$ by constructing an edge $v_i$ to $v_j$ with weight $A_{ij}$, whenever $A_{ij} \neq 0$.

Note that \reflem{IJCount} holds for weighted graphs, with ``the number of length $n$ paths'' interpreted as ``the total weight of the length $n$ paths.''

We may also state a general form of the Perron--Frobenius theorem. See \cite[Lemma VI.1.1]{minc:nonnegative}.

\begin{theorem}[General form of Perron--Frobenius]\label{Thm:PerronGeneral}
Let $\Gamma$ be a (weighted) directed graph with vertices $v_1, \ldots, v_k$, where $v_1$ is the start state. Then there is a reordering of $v_2, \ldots, v_k$, corresponding to a permutation matrix $P$, which conjugates $A = A(\Gamma)$ to a matrix $P^{-1} A P$ with the following properties:
\begin{enumerate}
\item $P^{-1} A P$ is block upper-triangular.
\item The diagonal blocks $B_1, \ldots, B_m$ of $P^{-1} A P$ are irreducible, and correspond to the strongly connected components of $\Gamma$.
\item The spectrum of $A$ is the union of the spectra of the diagonal blocks $B_1, \ldots, B_m$. In particular, $A$ has a Perron--Frobenius eigenvalue $\rho_A = \max \{ \rho_{B_i} \}$.
\end{enumerate}
\end{theorem}

An irreducible block $B_i$ is called \emph{maximal} if $\rho_{B_i} = \rho_A$. In this situation, the strongly connected component $\Gamma_i \subset \Gamma$ corresponding to $B_i$ is also called \emph{maximal}.

We will also need the following monotonicity result about Perron--Frobenius eigenvalues. See \cite[Corollary II.2.2]{minc:nonnegative}.

\begin{theorem}\label{Thm:Domination}
Let $A$ and $B$ be non-negative $k \times k$ matrices. If $A\leq B$ in the sense that $A_{ij}\leq B_{ij}$ for each $i,j$, then
$\rho_A\leq \rho_B$. Moreover, if $B$ is irreducible and $A\leq B$ but $A \neq B$, then $\rho_A<\rho_B$.
\end{theorem}

\subsection{Counting paths and words }
Let $L$ be a regular language with finite state automaton $\Gamma$. We allow the edges of $\Gamma$ to carry positive weights, which means that each
 word of $L$ also carries a weight, equal to the weight of the corresponding path in $\Gamma$.
Let $L_n$ denote the set of words of length exactly $n$, and $L_{\leq n}$  the set of words of length at most $n$. We let $w(L_n)$ denote the total weight of all the words of length $n$, and similarly for $w(L_{\leq n})$.
For a parallel with equation~\refeqn{fH}, define
\begin{equation}\label{Eqn:fL}
f_L(n) = w(L_{\leq n}).
\end{equation}
Note that if all the weights are $1$, then $f_L(n)$ is the number of words in $L_{\leq n}$.

The following result relates the growth of $w(L_n)$ to the Perron--Frobenius eigenvalue of $A(\Gamma)$.

\begin{proposition}\label{Prop:PolyExp}
Let $L$ be an infinite regular language with a pruned, weighted automaton $\Gamma$.
Let $A$ be the adjacency matrix of $\Gamma$. Then there is a number $0 < \tau < \rho_A$ and a positive integer $h$, such that for each $s\in\{1,\ldots, h\}$ there is a polynomial $\pi_s(n)$ with
\begin{equation}\label{Eqn:PolyExp}
w( L_{ n} ) = \pi_s(n) \, \rho_A^n + O(\tau^n),
\end{equation}
for each $n\in \naturals$ with  $n \equiv s \mod h$. Furthermore, $\pi_s(n) \neq 0$ for at least one $s$.
\end{proposition}

Similar statements appear in \cite[Theorem V.3]{flajolet-sedgewick} and \cite[Proposition 3.1.4]{calegari:ergodic-groups}. However, those results do not identify the exponential growth rate of $w(L_n)$ as the Perron--Frobenius eigenvalue $\rho_A$.

To avoid breaking up the exposition, we postpone the proof of \refprop{PolyExp} to \refsec{GrowthRate}. For now, we derive the following important consequence.

\begin{theorem}\label{Thm:PathCount}
Let $L$ be an infinite regular language with a pruned, weighted automaton $\Gamma$.
Let $A$ be the adjacency matrix of $\Gamma$.
Then 
\[
\lim_{n \to \infty} \sqrt[n]{ f_L(n) } = \lim_{n \to \infty} \sqrt[n]{ w( L_{\leq n} )} = \max \,  \{  \rho_A, \, 1 \}.
\]
\end{theorem}

\begin{proof}
Suppose, as a warm-up case, that $h = 1$ in \refprop{PolyExp}. That is, suppose there is a single nonzero polynomial $p(n)$ such that
\[ w(L_n) = p(n)  \rho_A^n + O(\tau^n),\]
where $\tau < \rho_A$. In this case, there is a different polynomial $q(n)$, of the same degree as $p(n)$, such that
\begin{equation}\label{Eqn:LnSum}
w(L_{\leq n}) = \sum_{i = 0}^n w(L_i) =  q(n)  \rho_A^n + O(\tau^n) + O(1).
\end{equation}
One way to derive \refeqn{LnSum} is to apply the method of summation by parts. Another way is to approximate the sum as an integral, and perform integration by parts:
\[
\sum_{i = 0}^n w(L_i) \sim \sum_{i = 0}^n p(i) \rho_A^i \sim \int_0^n p(x) \rho_A^x \, dx =  q(n) \rho_A^n -  q(0) \rho_A^0.
\]

In general,   \refeqn{PolyExp} expresses $w(L_n)$  in terms of a polynomial $\pi_s(n)$ that depends on $s \equiv n \mod h$.
 Summing together $h$ consecutive terms of \refeqn{PolyExp} removes this dependence:
\[
\sum_{i=rh+1}^{rh+h} w(L_{i}) = p(rh)  \rho_A^{rh} + O(\tau^{rh}),
\]
where $p$ is a nonzero polynomial.
Therefore, when we calculate $w(L_{\leq n})$ for $n = rh+s$, the sum of the first $rh$ terms is independent of $s$. Thus we get the following analogue of \refeqn{LnSum}:
\begin{equation}\label{Eqn:LnSumGeneral}
w(L_{\leq n}) = \sum_{i = 0}^{rh} w(L_i) + \sum_{i = rh+1}^{rh+s} w(L_i) =  q_s(n)  \rho_A^n + O(\tau^n) + O(1),
\end{equation}
where $\deg q_s(n) = \deg p(n)$, hence $q_s(n) \neq 0$.

We now analyze the asymptotics as $n \to \infty$. If $\rho_A \geq 1$, the term $q_s(n)  \rho_A^n$ is dominant in \refeqn{LnSumGeneral}.
 Therefore,
\[
\lim_{n \to \infty} \sqrt[n]{ w(L_{\leq n})} = \lim_{n \to \infty}  \sqrt[n]{ q_s(n) \rho_A^n } =
 \rho_A.
\]
Otherwise, if $\rho_A < 1$, if follows that $\tau < 1$. Hence the dominant term in \refeqn{LnSumGeneral} will be $O(1)$, and
\[
\lim_{n \to \infty} \sqrt[n]{ w(L_{\leq n})} = \lim_{n \to \infty} \sqrt[n]{  O(1)}  = 1. \qedhere
\]
\end{proof}

An immediate consequence of Theorems~\ref{Thm:Automatic} and \ref{Thm:PathCount} is the following version of Cannon's theorem on growth rates:

\begin{corollary}\label{Cor:CannonGrowth}
Let $G$ be an infinite hyperbolic group, with finite symmetric generating set $S$. Let $L$ be a geodesic regular language mapping bijectively to $G$. Then
\[
\lambda_G =  \rho_A
 =
\lim_{n \to \infty} \sqrt[n]{ f_L(n) }
 =  \lim_{n \to \infty} \sqrt[n]{ f_{G,\Upsilon}( n) } ,
\]
where $\rho_A$ is the Perron--Frobenius eigenvalue of any pruned automaton for $L$.
\end{corollary}

\begin{proof}
Any pruned automaton $\Gamma$ for the language $L$ has weights of $1$ on the edges, hence \refeqn{fH} and \refeqn{fL} give
\[
f_{G,\Upsilon}(n) = f_L(n) = w(L_{\leq n}).
\]
Since $A(\Gamma)$ is an integer matrix, we have $\rho_A \geq 1$. Now, \refthm{PathCount} gives the result.
\end{proof}


\section{Growth rates of regular languages}\label{Sec:GrowthRate}

The goal of this section is to prove \refprop{PolyExp}, which is needed in the proof of \refthm{PathCount}.
It is worth mentioning that \cite{flajolet-sedgewick} contains special cases of the same statement: see Theorem V.3 and Proposition V.7. While the proofs in \cite{flajolet-sedgewick} are rooted in complex analysis, we will derive \refprop{PolyExp} from the Perron--Frobenius theorem plus elementary facts about non-negative matrices.

\begin{lemma}\label{Lem:PowerIteration}
Let $A$ be an irreducible,  aperiodic, $k \times k$ matrix, and  $V, W \in \RR^k$. Assume that $A,V,W$ are non-negative and nonzero. Then there are constants $C >0$ and $0 < \tau < \rho_A$ such that
\[
  V^T \! A^n W = C \rho_A^n + O(\tau^n).
\]
\end{lemma}

\begin{proof}
Since $A$ has period $1$, \refthm{PerronIrreducible}\refitm{Period} implies there is only one eigenvalue with absolute value $\rho_A$. By \refthm{PerronIrreducible}\refitm{Eigenspace}, the  eigenspace of $\rho_A$ is spanned by a single unit-length eigenvector $E$, all of whose entries are positive. Thus $W$ has a positive projection to $E$. Under these hypotheses, the method of power iteration (see e.g.\ \cite[Section 7.3.1]{golub-vanloan})
 produces a convergent sequence
\[
\frac{A^n W}{||  A^n W ||} \longrightarrow E,
\]
with an exponential rate of convergence. Thus, for $n \gg 0$, the vector $A^n W$ has all positive entries, and these entries grow by a factor converging exponentially quickly to $\rho_A$. Consequently the product $V^T (A^n W)$ also grows by a factor converging exponentially quickly to $\rho_A$.
\end{proof}

\begin{lemma}\label{Lem:Unipotent}
Let $U$ be an $m \times m$ upper triangular matrix with $1$'s on the diagonal. Then, for $n \in \naturals$, the entry $(U^n)_{ij}$ is given by a polynomial  $p_{ij}(n)$. Furthermore, the degree of $p_{ij}$ is bounded above by $|j-i|$.
\end{lemma}

\begin{proof}
Write $U^n = ((U-I)+I)^n$, and apply the binomial theorem. Since $(U-I)$ is nilpotent, all terms above degree $m$ in $(U-I)$ will vanish.
\end{proof}

\begin{lemma}\label{Lem:Jordan}
Let $J$ be a $k \times k$ matrix in Jordan form. Suppose that the first $m $ diagonal  entries are $1$, and that any other diagonal entries have absolute value less than $\tau < 1$. Then, for $n \in \naturals$ and for any row vector $W \in \RR^k$, we have
\begin{equation}\label{Eqn:JordanProduct}
W J^n = (p_1(n), \ldots, p_m(n), O(\tau^n), \ldots, O(\tau^n)),
\end{equation}
where  each $p_i(n)$ is a polynomial of degree at most $i-1$.
\end{lemma}

\begin{proof}
The top $m \times m$ block of $J$ is a unipotent matrix $U$ as in \reflem{Unipotent}. Since $J$ is in Jordan form, the first $m$ entries of $W J^n$ are given by multiplying the first $m$ entries of $W$ by $U^n$. By \reflem{Unipotent}, these first $m$ entries are polynomials of the indicated degrees. The remaining  entries decay as $\tau^n$ because the remaining Jordan blocks have eigenvalues less than $\tau$.
\end{proof}

\begin{lemma}\label{Lem:Aperiodic}
Let $A$ be a $k \times k$ non-negative matrix with Perron--Frobenius eigenvalue $\rho_A > 0$. Suppose that any eigenvalue of absolute value $\rho_A$ must actually equal $\rho_A$. Then, for any $V \in \RR^k$ and $n \in \naturals$,
\[
V (A/\rho_A)^n = (p_1(n) + O(\tau^n) \ , \  \ldots \ ,  \  p_k(n) + O(\tau^n)),
\]
where $\tau < 1$ and $p_i(n)$ is a (possibly zero) polynomial in $n$.
\end{lemma}

\begin{proof}
Note that $(A/\rho_A)$ has eigenvalue 1 with multiplicity $m \geq 1$, and all other eigenvalues have absolute value bounded by $\tau < 1$. Thus $(A/\rho_A)$ is conjugate to a matrix $J$ satisfying \reflem{Jordan}. Hence there is an invertible matrix $R$ such that
\[
V (A/\rho_A)^n = V (R^{-1} J R)^n = V R^{-1} J^n R.
\]
Applying \reflem{Jordan} to $W = V R^{-1}$, we see that $W J^n$ has the form given in \refeqn{JordanProduct}, with $m$ polynomials in the leading entries and $O(\tau^n)$ in the remaining entries. Consequently each entry of $(W J^n) R$ is of the form $p_i(n) + O(\tau^n)$ for new polynomials $p_i$.

Finally, in the special case where all eigenvalues of $J$ are $1$, then there are no $O(\tau^n)$ terms, and we may thus choose an arbitrary $0<\tau < 1$ to satisfy the statement of the lemma.
\end{proof}

We can now complete the proof of \refprop{PolyExp}.

\begin{proof}[Proof of \refprop{PolyExp}]
The proof proceeds in two steps. In Step~1, we prove that there is an integer $h$ such that
\begin{equation}\label{Eqn:PolyExpRestate}
w( L_{ n} ) = \pi_s(n) \, \rho_A^n + O(\tau^n),
\end{equation}
where $\pi_s$ is a polynomial depending on   $s \equiv n \mod h$.
This can be viewed as an upper bound on the exponential growth rate  of $w(L_n)$, which is attained if and only if $\pi_s(n) \neq 0$. In Step~2, we analyze a sub-language $L' \subset L$ and show that  it grows at least as fast as $\rho_A^n$. This gives a lower bound on $w(L_n)$ and ensures that $\pi_s(n) \neq 0$ for some $s$.

Assume, without loss of generality, that the vertices $v_1, \ldots, v_k$ of $\Gamma$ have been reordered as in \refthm{PerronGeneral}, hence $A$ is in block upper-triangular form with irreducible blocks $B_1, \ldots, B_m$. The hypothesis that $L$ is infinite ensures that every maximal component $\Gamma_i \subset \Gamma$ contains a nontrivial closed directed path, which implies that each maximal block $B_i$ is nonzero, hence $\rho_A = \rho_{B_i} > 0$ by \refthm{PerronIrreducible}.

Let $h(B_i)$ be the period of the $i$-th block. We choose a positive integer $h$ that is a multiple of each $h(B_i)$, and furthermore such that $h$ is (a multiple of) the length of some closed directed path based at some vertex $v_\ast$ in a maximal component.
By \refthm{PerronIrreducible}\refitm{Period}, every eigenvalue of $A^h$ that has absolute value $\rho_A^h$ must actually \emph{equal} $\rho_A^h$, hence $A^h$ satisfies \reflem{Aperiodic}.

\smallskip

\underline{\emph{Step 1.}}
Let $V_0 \in \RR^k$ be the row vector $(1, 0, \ldots, 0)$.
Let $V_a \in \RR^k$ be the column vector whose $j$-th entry is $1$ if $v_j$ is an accept state, and $0$ otherwise.
By \reflem{IJCount}, the total weight of the length $n$ paths in $\Gamma$ from $v_1$ to $v_j$ is $(A^n)_{1j} = (V_0 A^n)_j$. Thus the total weight of length $n$ paths from $v_1$ to accept states is
\begin{equation}\label{Eqn:LnFormula}
w(L_n) = V_0 A^n V_a .
\end{equation}

Fix  $s \in\{1,\ldots, h\}$ and suppose that $n = rh+s$ for $r \in \naturals$. By \reflem{Aperiodic}, there exist polynomials $p_i, q_i$ and constants $\sigma, \sigma' < 1$ such that
\begin{align*}
\frac{ V_0 A^n}{\rho_A^n} = \frac{V_0 A^s}{\rho_A^s} \cdot \left( \frac{A^h}{\rho_A^h} \right)^r
& = (p_1(r) + O((\sigma')^r), \ldots, p_k(r) + O((\sigma')^r) ) \\
& = (q_1(n) + O(\sigma^n), \ldots, q_k(n) + O(\sigma^n) ).
\end{align*}
The first equality above holds because  $n = rh+s$, the second equality holds by \reflem{Aperiodic}, and the third equality holds by letting  $\sigma' = \sigma^h$ and noting that polynomials in $r$ are also polynomials in $n$.

Right-multiplying by $V_a$, we obtain the following equality
for some polynomial $\pi_s$ and some constant  $\sigma < 1$. This is equivalent to \refeqn{PolyExpRestate}.
\[
\frac{w(L_n)}{\rho_A^n} = \frac{V_0 A^n V_a}{\rho_A^n} = \pi_s(n) + O(\sigma^n) .
\]

\smallskip

\underline{\emph{Step 2.}}
It remains to show that $\pi_s(n) \neq 0$ for at least one $s$. To that end, we will construct a sub-language $L' \subset L$, which grows roughly as quickly as $\rho_A^n$. More precisely, we find $s \in\{1,\ldots, h\}$ and constants $C > 0$ and $\tau < \rho_A$, such that
\begin{equation}\label{Eqn:SublangCount}
w(L_n) \ \geq \ w(L'_n) \ = \ C \rho_A^n + O(\tau^n)
\qquad \mbox{for} \quad n \equiv s \mod h.
\end{equation}
Comparing  \refeqn{PolyExpRestate} to \refeqn{SublangCount}, it follows that $\pi_s(n) \neq 0$ for the corresponding $s$.

By the definition of $h$, there is a length~$h$ closed directed path based at  a state $v_\ast$ belonging to a maximal component $\Gamma_i \subset \Gamma$. Since  $\Gamma$ is pruned, there is a directed path $\gamma$ from $v_1$ to $v_\ast$ and a directed path $\delta$ from $v_\ast$ to an accept state. Let
\[
s \equiv \ell(\gamma) + \ell(\delta) \mod h.
\]

Let $L' \subset L$ be the sub-language corresponding to paths in $\Gamma$ that follow $\gamma$ from $v_1$ to $v_\ast$, then follow closed directed paths of length $rh$ based at $v_\ast$ (for some $r \in \naturals$), then follow $\delta$ to an accept state. Note that the closed paths based at $v_\ast$ must  lie in $\Gamma_i$. By construction, every word in $L'$ has length $n \equiv s \mod h$.

Let $B_i$ be the maximal irreducible block of $A$ corresponding to $\Gamma_i$. The matrix $B_i^h$ may not be irreducible, but by \refthm{PerronGeneral} it contains an irreducible block $D$ corresponding to a weighted subgraph of $\Gamma_i$ containing $v_\ast$. In addition, every eigenvalue of $B_i^h$ with absolute value $\rho_{B_i}^h = \rho_A^h = \rho_D$ must actually \emph{equal} $\rho_A^h$, hence $D$ is both irreducible and aperiodic.

By \reflem{IJCount}, some diagonal entry  $(D^r)_{jj}$ is the total weight of the length~$rh$ directed closed paths based at $v_\ast$. Let $W$ be a vector with $1$ in the $j$-th entry and $0$'s elsewhere. By \reflem{PowerIteration},  the total weight of the length $rh$ paths based at $v_\ast$ is
\[
(D^r)_{jj} = W^T D^r W = (C') \rho_D^r + O((\tau')^r) =  (C') \rho_A^{rh} + O((\tau')^r)
\]
for constants $C' > 0$ and $\tau' < \rho_A^h$.

Since words in $L'$ of length $\ell(\gamma) + \ell(\delta) + rh$ are in $1$--$1$ correspondence with closed directed paths at $v_\ast$ of length $rh$, we have
\[
w(L'_{\ell(\gamma) + \ell(\delta) + rh}) = w(\gamma) \cdot w(\delta) \cdot W^T D^r W,
\]
where $w(\gamma)$ and $w(\delta)$ are the weights of $\gamma$ and $\delta$ respectively. Thus, setting $n = \ell(\gamma) + \ell(\delta) + rh$, we obtain
\[
w(L_n) \geq w(L'_n) =  w(\gamma) \cdot w(\delta) \cdot (C') \rho_A^{rh} + O((\tau')^r) = C \rho_A^n + O(\tau^n),
\]
where $\tau' = \tau^h$. This establishes \refeqn{SublangCount}, completing the proof.
\end{proof}

\section{The torsion-free case}\label{Sec:TorsionFree}

This section gives a quick proof of the following special case of \refthm{GrowthGeneral}:

\begin{theorem}\label{Thm:GrowthTorsionFree}
Let $G$ be a non-elementary, torsion-free hyperbolic group with generating set $S$. Let $\Upsilon = \Upsilon(G,S)$ be the Cayley graph of $G$ with respect to $S$.
Let $H$ be a quasiconvex subgroup of infinite index. Then
\[
\lambda_H(\Upsilon) < \lambda_G(\Upsilon).
\]
\end{theorem}

In addition to the background in Sections~\ref{Sec:Automatic} and \ref{Sec:Perron}, the proof of \refthm{GrowthTorsionFree} uses the following theorem first formulated by Gromov \cite{Gromov87}.

\begin{theorem}[Free product]\label{Thm:FreeProduct}
Let $G$ be a non-elementary,  torsion-free hyperbolic group. Let $H$ be an infinite index quasiconvex subgroup. Then $\exists g \neq 1$  such that $\langle H, g \rangle \cong H \ast \langle g \rangle$.
\end{theorem}

 See
Arzhantseva \cite[Theorem~1]{Arzhantseva2001}  for a proof, and see Gitik  \cite[Corollary~4]{Gitik99} for a similar statement with additional hypotheses. See also \refthm{FreeProductAlternate}, which gives a slight generalization using ping-pong on $\bdy G$.
Note that Arzhantseva also proves the stronger result that
 $\langle H, g^m \rangle$ is quasiconvex in $G$ for sufficiently large $m$.

\begin{proof}[Proof of \refthm{GrowthTorsionFree}]
Assume that $H$ is non-trivial, as otherwise the statement of the theorem is immediate since $G$ has exponential growth. By
\refthm{FreeProduct}, we may choose an element $g \neq 1$ such that $K = \langle H, g \rangle \cong H \ast \langle g \rangle$.

By \refthm{Automatic}, let $L_G$ be a regular language of geodesics in $\Upsilon(G,S)$,  mapping bijectively to $G$. By \refthm{QuasiconvexRational}, the sub-language $L_H$ consisting of geodesics words mapping to points of $H$ is itself regular.

Let $\Gamma_H$ be a pruned finite state automaton that accepts $L_H$. Let $\rho_H$ be the Perron--Frobenius eigenvalue of the adjacency matrix of $\Gamma_H$.
By \refcor{CannonGrowth}, we have
\begin{equation}\label{Eqn:HGrowth}
\lambda_H =  \rho_H.
\end{equation}

Let $\sigma$ be an arc in $\Upsilon$ from $1$ to $g $. Let $x$ be the label of the first edge of $\sigma$. We introduce a new letter $x'$ into our alphabet, with the understanding that $x'$ maps to $x \in S$ when words are mapped to group elements. Let $\sigma'$ be a copy of $\sigma$, with $x$ replaced by $x'$.

Let $\Gamma_M$ be a finite state automaton built from $\Gamma_H$ as follows: for each accept state  $v_i \in \Gamma_H$, we attach an arc from $v_i$ to the start state $v_1$,
in the form of a directed, labeled copy of $\sigma'$. (The replacement $x \to x'$ ensures that $\Gamma_M$ is deterministic.)

We claim that $\Gamma_M$ is strongly connected. Indeed, every state of $\Gamma_H$ is reachable from the start state, leads to an accept state, and the accept state leads to the start via $\sigma'$. Thus every vertex (including the vertices on the copies of $\sigma'$) is part of a directed closed path through the start state $v_1$.

The language $L_M$ accepted by $\Gamma_M$ consists of words mapping to the monoid $M$
generated by $H$ and positive powers of $g$. Since $K$ is a free product, the composed map $L_M  \to M \to G$ is injective. Furthermore, a word of length $n$ in $L_M$ maps to a path of length $n$, hence the endpoint of this path lies in the ball of radius $n$ about $1 \in \Upsilon$. Thus, letting $f_M(n)$ be the number of words of length at most $n$ in $L_M$, as in \refeqn{fL}, we have
\begin{equation}\label{Eqn:KGrowth}
f_M(n) \leq f_{G,\Upsilon}(n).
\end{equation}

Therefore, we have
\[
\lambda_H =  \rho_H <  \rho_M = \lim_{n \to \infty} \sqrt[n]{ f_M(n)} \leq \lim_{n \to \infty} \sqrt[n]{ f_{G,\Upsilon}(n)} = \lambda_G.
\]
Here, the first equality holds by \refeqn{HGrowth}. The strict inequality holds by \refthm{Domination}, since $\Gamma_H$ is a proper subgraph of the strongly connected graph $\Gamma_M$. The next equality holds by \refthm{PathCount}. The non-strict inequality holds by \refeqn{KGrowth}, and the final equality is by \refeqn{LambdaDef}, the definition of $\lambda_G$. Note that the limit exists by \refcor{CannonGrowth}.
\end{proof}

\section{General actions by hyperbolic groups}\label{Sec:GeneralHyp}

Proving \refthm{GrowthGeneral} in the general case of group actions on graphs requires dealing with several complexities that did not arise in \refsec{TorsionFree}. The next two subsections give a way to circumvent these complexities.
First, \refthm{GroupoidLanguage} gives an analogue of \refthm{Automatic} and \refcor{CannonGrowth} for group actions on graphs that may have multiple vertex orbits and non-trivial point stabilizers. Next, \refthm{FreeProductAlternate} gives an analogue of \refthm{FreeProduct} that will work in the presence of torsion. With these results in hand, we can complete the proof of \refthm{GrowthGeneral}.

\subsection{A language for group actions}\label{Sec:ConstructAction}

Let $G$ be a group acting properly and cocompactly on a graph $\Upsilon$. The following constructions build a regular language $L_G$ adapted to this action. The results are summarized in \refthm{GroupoidLanguage}.

\begin{construction}[Free Action]\label{Const:FreeAction}
Let $G$ act on a graph $\Upsilon$. We may assume without loss of generality that $G$ acts without inversions. For,  if $G$ inverts an edge $e$, we add a second copy of $e$ without changing any distances in $\Upsilon$. We retain the name $\Upsilon$.

We construct a new graph  $\widehat \Upsilon$ with a free action by $G$. To that end, choose representatives of the orbits in $\Upsilon$ of vertices and edges.

Each vertex of $\widehat \Upsilon$ is a pair $(g,v)$, where $g \in G$ and $v$ is a representative vertex of $\Upsilon$.
An edge of $\widehat \Upsilon$ is likewise a pair $(g,e)$, where $e$ is a representative edge of $\Upsilon$.
The edge $(g,e)$ connects vertices $(gh, u )$ and $(gk, v)$ in $\widehat \Upsilon$ whenever   $e$ connects vertices $h u$  and $k v$ in $\Upsilon$.
Note that $G$ acts freely on $\widehat \Upsilon$ and that there is an equivariant surjection $\widehat \Upsilon \rightarrow \Upsilon $ induced by $(g,v) \mapsto gv$ and $(g,e) \mapsto ge$.

For each vertex $v$ of $\Upsilon$, let $\stab_G(v)$ denote its stabilizer.
We now form a new graph $\Upsilon^*$ as follows. For every representative vertex $v$, and every left coset $g \stab_G(v)$, we connect every pair of elements of  $(g \stab_G(v) ) (1,v)$ by an edge. This includes loop edges with both endpoints at $g(1,v)$.
We refer to these new edges as \emph{tiny edges}. Add duplicates of tiny edges corresponding to order $2$ elements of $\stab_G(v)$, ensuring that $G$ acts on $\Upsilon^*$ without inversions.
Again, there is an equivariant surjection $\Upsilon^*\rightarrow \Upsilon$, which collapses every tiny edge.
We call $\Upsilon^*$ the \emph{blowup} of $\Upsilon$. Note that $G$ acts freely on $\Upsilon^*$.

The point of adding tiny edges is that without them, $\widehat \Upsilon$ may not be connected; see \refex{Dihedral}. However, we have the following.

\begin{claim}\label{Claim:Connected}
If $\Upsilon$ is connected, then $\Upsilon^*$ is also connected.
\end{claim}

\begin{proof} A path $e_1 \cdots e_n$ in $\Upsilon$ lifts to a sequence $\hat{e}_1, \ldots, \hat{e}_n$ of edges in $\widehat{\Upsilon}$. Letting $v_i$ denote the vertex between  $e_i$ and $e_{i+1}$, the terminal vertex of
$\hat{e}_i$ lies in the same $\stab_G(v_i)$ orbit as the initial vertex of $\hat{e}_{i+1}$. We may thus join them by tiny edges to create a path in $\Upsilon^*$.
\end{proof}
\end{construction}

\begin{example}\label{Ex:Dihedral} Let $G = \ZZ_2 * \ZZ_2 = \langle a,b \mid a^2, b^2\rangle$. Let $\Upsilon$ be the Bass--Serre tree of the free product. That is: $\Upsilon$ is a copy of $\RR$, with vertices at $\ZZ$, on which $a$  acts by reflection about $0$ and $b$ acts by reflection about $1$. There are two $G$--orbits of vertices (namely, even and odd integers), and every vertex is stabilized by a conjugate of $\langle a \rangle$ or $\langle b \rangle$. Every edge is in the $G$--orbit of $e = [0,1]$. See Figure~\ref{Fig:Blowup}.

Given this setup, the $0$--skeleton of $\widehat \Upsilon$ is  $\widehat \Upsilon^{(0)} = \ZZ \times \{0,1\}$. Then $a$ and $b$ act on $\widehat \Upsilon^{(0)}$ by reflecting each copy of $\ZZ$ (about $0$ and $1$, respectively) and then interchanging the two copies. The combined effect appears as a rotation in Figure~\ref{Fig:Blowup}. Thus, for every integer $n$, $\widehat \Upsilon$ has an edge of the form $((ba)^n, e)$ with vertices at $(2n,0)$ and $(2n\!+\!1,0)$. Similarly, $\widehat \Upsilon$ has an edge of the form $(a(ba)^n, e)$ with vertices at $(-2n,1)$ and $(-2n\!-\!1,1)$. In particular, $\widehat \Upsilon$ has infinitely many connected components. 

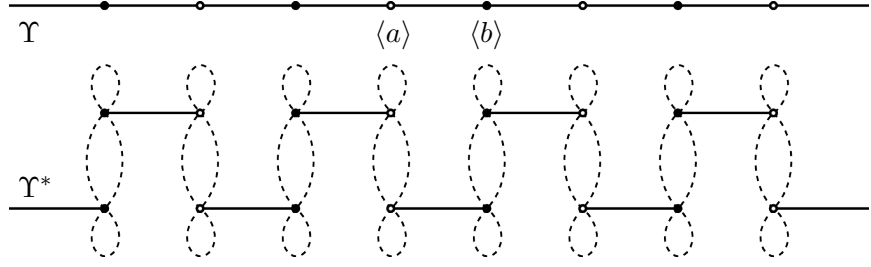
\begin{figure}

\begin{overpic}{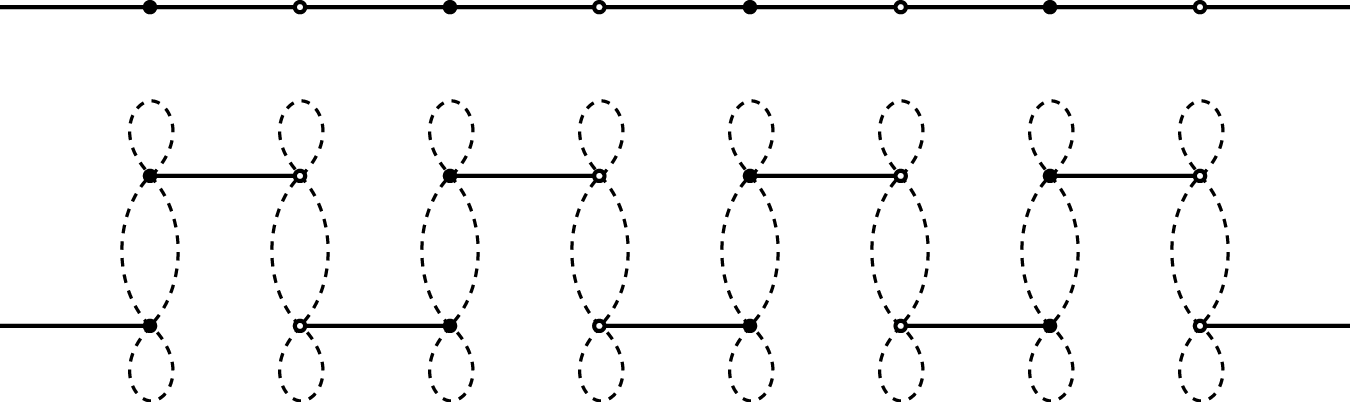}
\put(1,25){$\Upsilon$}
\put(1,7){$\Upsilon^*$}
\put(42.5,25){$\langle a \rangle$}
\put(53.5,25){$\langle b \rangle$}
\end{overpic}

\caption{The graphs of \refex{Dihedral}. Top: the Bass--Serre tree $\Upsilon$, where $\langle a \rangle$ and $\langle b \rangle$ stabilize vertices as shown. Bottom: the blowup $\Upsilon^*$, with tiny edges shown dashed. Deleting the tiny edges gives $\widehat \Upsilon$. Note that $\langle a \rangle$ and $\langle b \rangle$ act on $\Upsilon^*$ by rotation, without fixed points.}
\label{Fig:Blowup}
\end{figure}


To form the connected graph $\Upsilon^*$, we add the following tiny edges: one loop edge at every vertex of $\widehat \Upsilon$, as well as \emph{two} tiny edges connecting $(v,0)$ to $(v,1)$ for every $v \in \ZZ$. The two tiny edges from $(v,0)$ to $(v,1)$ are permuted by $\stab_G(v)$.

\end{example}

\begin{construction}[Transitive Action] \label{Const:TransitiveAction}
Let $G$ act cocompactly (and without inversions) on a graph $\Upsilon$. We will create a new group $G^+$ acting transitively on the vertices of a graph $\Upsilon^+$.

Let $\Upsilon^*$ be the blowup of $\Upsilon$, with the resulting free $G$--action,
as described in Construction~\ref{Const:FreeAction}.
We attach $2$--cells as follows. First,
choose a single representative from each $G$--orbit of based cycles, and attach a $2$--cell along it. Then extend equivariantly,  to obtain a simply connected $2$--complex with a free $G$ action.
We retain the name $\Upsilon^*$.
Let $D$ be the quotient of $G\backslash \Upsilon^*$ obtained by identifying all $0$--cells.
Then $ \pi_1 D  = G^+ \cong G * F_r$, where   $F_r$ is a free group whose generators  are in $1$--$1$ correspondence with edges in a spanning tree for $G\backslash \Upsilon^*$.

Consider the universal cover $\widetilde D$, which is a tree of copies of $\Upsilon^*$. We let $\Upsilon^+$ be the $1$--skeleton of $\widetilde{D}$. The deck group $G^+$ acts transitively on the vertices of $\Upsilon^+$.
\end{construction}

\begin{construction}[Regular Language]\label{Const:GroupoidLanguage}
We continue with the notation of Construction~\ref{Const:TransitiveAction}.
 Choose a generating set $S^+$ for  $G^+$ by considering its action on $\Upsilon^+$. Each generator   corresponds to a closed path in the $1$--skeleton of $D$ of the form $yey'$, where $e$ is a non-tiny edge, and $y, y'$  are tiny edges.
This includes the case where $y$ or $y'$ is a loop edge in $\Upsilon^*$, representing a trivial element of $G$.

Assuming that $\Upsilon$ has at least one edge, each tiny edge is homotopic to the concatenation of two generator paths. Thus the proof of \refclaim{Connected} shows that
$S^+$ generates $G^+$.
The set $S^+$ is finite whenever the action of $G$ on $\Upsilon$ is proper and cocompact.
Note that $S^+$ is symmetric by definition.

Since  $G^+ \cong G * F_r$ is hyperbolic, \refthm{Automatic} provides a geodesic regular language $L^+$ that maps bijectively to $G^+$. Let $L_G \subset L^+$ be the sublanguage mapping bijectively to $G \subset G^+$.
Since $G$ is quasiconvex in $G^+ \cong G * F_r$, the sublanguage $L_G$ is regular by \refthm{QuasiconvexRational}.


Each word in $L^+$ is a path in $\Upsilon^+$ starting at the canonical basepoint $b$.
Since $\Upsilon^+$ is a tree of copies of $\Upsilon^*$, the words of $L_G \subset L^+$ correspond to paths that stay in one copy of $\Upsilon^*$. Consider the projection  $\Upsilon^* \to \Upsilon$, which collapses all tiny edges.
\end{construction}

Let $L_b = L_{\stab_G(b)}$ be the finite set of words in $L_G$ mapping to $b \in \Upsilon$.

\begin{claim}\label{Claim:SameLength}
The projection  $\Upsilon^* \to \Upsilon$ maps words of length $n$ in $L_G - L_b$ to length $n$ geodesics in $\Upsilon$.
\end{claim}

\begin{proof}
Since every generator in $S^+$ contains exactly one non-tiny edge, a word of length $n$ always determines a path of length $n$. What needs to be shown is the converse: a geodesic of length $n \geq 1$ in $\Upsilon$ is always hit by a word of length $n$ in $L_G$.

A geodesic  in $\Upsilon$ from $b$ to $gb \neq b$ is a path $e_1 e_2\cdots e_n$. By \refclaim{Connected}, the geodesic in $\Upsilon$
``lifts'' to a path of the form $y_0 e_1 y_1 e_2 y_2 \cdots e_n y_n$ in $\Upsilon^*$.
This path determines a word  $(y_0e_1y_1)(y'_2 e_2y_2)\cdots (y'_n e_n y_n)$ in $S^+$ of the same length, where every $y'_i$ is a loop edge. By the previous paragraph, a word of length less than $n$ is not possible.
Hence the geodesic language $L_G$ contains a word of length $n$ mapping to this path.
\end{proof}

The result of these constructions is encapsulated in the following theorem.

\begin{theorem}\label{Thm:GroupoidLanguage}
Let $G$ be a hyperbolic group acting properly and cocompactly on a graph $\Upsilon$. Fix a basepoint $b \in \Upsilon$. Then there is a regular language $ L_{G}$ with the following properties.
\begin{enumerate}
\item\label{Itm:FiberSize} $L_G \to Gb$ is a surjection with fibers of cardinality exactly $|\stab_G(b)|$.
\item\label{Itm:Lb} The words in $L_b$, i.e.\ the preimage of $b$, have length $1$ or $0$, and correspond to paths of length $0$.
\item\label{Itm:SameLength}  Every word of length $n$ in $L_G - L_b$ corresponds to a length $n$ geodesic in $\Upsilon$, starting at $b$. \item\label{Itm:Sublanguage} For every quasiconvex subgroup $H \subset G$, the sublanguage $L_H \subset L_G$ of words mapping to $Hb$ is regular.
\item\label{Itm:SublangGrowth}
Let $\rho_H$ be the Perron--Frobenius eigenvalue of the transition matrix for any pruned automaton accepting $L_H$.
Then growth rate $\lambda_H(\Upsilon)$ satisfies
\begin{equation}\label{Eqn:GroupoidGrowth}
\lambda_H(\Upsilon) =
 \lim_{n \to \infty} \sqrt[n]{  f_{H,\Upsilon}(n)} =
 \lim_{n \to \infty} \sqrt[n]{  f_{L_H}(n) } =  \rho_H.
\end{equation}

\end{enumerate}
\end{theorem}

\begin{proof}
Recall, from Construction~\ref{Const:GroupoidLanguage}, that $L_G$ is a regular sublanguage of the language $L^+$.
Since $L_G$ maps bijectively to $G$, and the surjection $G \to Gb$ has fibers of cardinality $|\stab_G(b)|$, conclusion \refitm{FiberSize} follows.

Conclusion \refitm{Lb} recalls the definition of $L_b$, combined with the fact that every non-trivial word of $L_b$ is expressible by a single letter in $S^+$.
Conclusion \refitm{SameLength} is a restatement of \refclaim{SameLength}.

Every quasiconvex subgroup $H \subset G$ is also quasiconvex in $G^+ \cong G * F_r$. Thus, by \refthm{QuasiconvexRational}, the sublanguage $L_H \subset L^+$ of words mapping to $H$ is also regular.

Equation~\refeqn{GroupoidGrowth} should be considered right to left. The right-most equality is by \refthm{PathCount}, and implies that the limit exists. The middle equality  follows by \eqref{Itm:FiberSize}, because $ f_{L_H}(n) = |\stab_G(b)| \cdot  f_{H,\Upsilon}(n)$, and  the constant factor disappears in the limit. The left-most equality is by the definition of $\lambda_H(\Upsilon)$.
\end{proof}

\subsection{Ping-pong with torsion}\label{Sec:PingPong}

Recall that our proof of \refthm{GrowthTorsionFree} relies on \refthm{FreeProduct}, which produces a free product $H * \langle g \rangle \subset G$.
Such a product may fail to exist when $H$ has torsion.
For instance, let  $G=G'\times \integers_p$,
and let $H=H'\times \integers_p$ for some quasiconvex $H'\subset G'$.
Then for any $g\neq 1$,
the subgroup $\langle H,g\rangle$ will have a nontrivial center, and thus does not split as a (nontrivial) free product.

Although the exact statement of \refthm{FreeProduct} does not hold in general, we have the following generalization  to hyperbolic groups with torsion.

\begin{theorem}\label{Thm:FreeProductAlternate}
Let $G$ be a hyperbolic group. Let $H \subset G$ be a quasiconvex subgroup, and let $Z_+ \subset G$ be a maximal elementary subgroup that is not commensurable with a subgroup of $H$.
Then there exists a finite index subgroup $Z \subset Z_+$
such that $\langle H, Z\rangle$ is isomorphic to the amalgamated free product
$H*_{(H\cap Z)} Z$, where $H \cap Z$ is finite.
\end{theorem}

A version of \refthm{FreeProductAlternate}, with the additional hypothesis that $H \cap Z_+$ is separable, is due to Mart{\'{\i}}nez-Pedroza and Sisto \cite[Corollary 4]{martinez-pedroza-sisto}.
To complete the analogy with \refthm{FreeProduct}, they also show  that $\langle H, Z \rangle$ is quasiconvex in $G$ whenever $[Z_+ : Z]$ is sufficiently large.

\begin{remark}\label{Rem:profusion}
When $[G:H] = \infty$, subgroups $Z_+$ as in \refthm{FreeProductAlternate} are abundant.
Indeed, $\boundary G$ is the closure of attracting fixed points of loxodromic elements,
and so we can choose an element $g$ with $\lim_{n\rightarrow \pm \infty} g^n = p_{\pm \infty} \notin \boundary H$.
Let $Z_+ = \stab_G ( \{p_+, p_- \} )$. Then $\boundary Z_+ =  \{p_+, p_- \}$, hence $Z_+$ is maximal elementary.
\end{remark}

We will give a short alternate proof of \refthm{FreeProductAlternate} using the following version of the ping-pong lemma. See Gitik
 \cite{Gitik99} for a particularly simple proof.

\begin{lemma}[Ping-pong]\label{Lem:ping-pong}
Let $H,Z$ be subgroups of a group $G$ acting on a set $\Omega$, and suppose $[Z:(H\cap Z)]>2$.
Let $\Omega_H$ and $\Omega_Z$ be disjoint, nonempty subsets of $\Omega$ such that
$(H-Z)\Omega_H\subset \Omega_Z$ and $(Z-H)\Omega_Z\subset \Omega_H$.
Then $\langle H,Z\rangle \cong H*_{H\cap Z} Z$.
\end{lemma}

\begin{proof}[Proof of \refthm{FreeProductAlternate}]
Let $\Omega = \bdy G$. We will construct
 an open neighborhood $U$ of $\boundary Z_+$ such that
$hU\cap U=\emptyset$ for each $h\in H-Z_+$.
Since $H$ acts properly discontinuously on $\Omega -\boundary H$, there is  an open neighborhood $V$ of $\boundary Z_+$ such that $\{h_1,\ldots, h_m\}$ is the finite subset of $H$
 with $h_iV\cap V\neq \emptyset$. By making $V$ smaller if necessary, we ensure that this finite set coincides with
$\stab_H(\boundary Z_+)$.
 Now, define $U = \cap_{i=1}^m h_i V$, and observe that $U$ is  $\stab_H(\boundary Z_+)$--invariant.

Let  $\Omega_H=U$. Let $\Omega_Z$ be the compact set $\Omega -U$.
Since $Z_+$ is maximal, we have $Z_+ = \stab_G(\bdy Z_+)$, hence  $(H-Z_+) \cap \stab_H(\boundary Z_+) = \emptyset$. Thus, for  $h\in H-Z_+$, we have
\[
h\Omega_H \: = \: hU \: \subset \: \boundary G-U \: = \: \Omega_Z.
\]
Recall that $Z_+$ acts on $\Omega = \bdy G$ with north--south dynamics.
Thus, all sufficiently long translators in $Z_+$ will squeeze the compact set $\Omega_Z$ into any open neighborhood about $\bdy Z_+$. By the separability of $(H\cap Z_+) \subset Z_+$, there is a finite index subgroup $Z\subset Z_+$ that contains $H\cap Z_+$ but excludes the finitely many elements that fail to map $\Omega_Z$ into $U$. In other words, for $z\in   Z-H$, we have
\[ z \Omega_Z \subset U = \Omega_H.
\]
Now, \reflem{ping-pong} completes the proof.
\end{proof}

\subsection{Exponential growth discrepancy}\label{Sec:FinalProof}

We can now restate and prove the main theorem of this paper. The proof follows the same outline as that of \refthm{GrowthTorsionFree}, while incorporating the extra structure developed in this section.

\begin{named}{\refthm{GrowthGeneral}}
Let $G$ be a non-elementary hyperbolic group acting properly and cocompactly on a graph $\Upsilon$.
Let $H$ be a quasiconvex subgroup of infinite index. Then 
\[
\lambda_H(\Upsilon) < \lambda_G (\Upsilon).
\]
\end{named}

\begin{proof}Assume that $H$ is infinite, as otherwise the statement of the theorem is immediate since $G$ has exponential growth. By  \refrem{profusion}, choose an infinite order element $g$ such that $\langle g\rangle$ is not commensurable with a subgroup of $H$. By \refthm{FreeProductAlternate}, there exists $m>0$ such that $K= \langle H,g^m\rangle$ splits as an amalgamated free product
 $H *_F Z$ over a finite group $F$, where $[Z:\langle g^m\rangle]<\infty$.

Fix a basepoint $b \in \Upsilon$. By \refthm{GroupoidLanguage}, there is a regular language $L_G$ with a sequence of maps
\begin{equation}\label{Eqn:Surjections}
L_G \xrightarrow{\: \alpha \: } G \xrightarrow{\: \beta \:} Gb,
\end{equation}
where $\alpha$ is a bijection and $\beta$ has fibers of constant cardinality $C = |\stab_G(b)|$.
\refthm{GroupoidLanguage} also guarantees that the sublanguage $L_H$ mapping to $H$ is regular.

Let $\sigma$ be an arc in $\Upsilon$ from $b$ to $g^m b$. Let $x$ be the label of the first edge of $\sigma$. We introduce a new letter $x'$ into our alphabet, with the understanding that $\alpha(x') = \alpha(x) \in G$. In other words, $x'$ represents the same group element as $x$. Let $\sigma'$ be a copy of $\sigma$, with $x$ replaced by $x'$. For later use, we assign a weight of $\frac{1}{|F|}$ to the initial edge of $\sigma'$. All other edges have weight $1$.

Let $\Gamma_H$ be a pruned finite state automaton that accepts $L_H$.
Let $\Gamma_M$ be a finite state automaton built from $\Gamma_H$ as follows: for each accept state  $v_i \in \Gamma_H$, we attach an arc from $v_i$ to the start state $v_1$,
in the form of a directed, labeled copy of $\sigma'$. As in the proof of \refthm{GrowthTorsionFree}, these arcs ensure that $\Gamma_M$ is strongly connected. As in \refeqn{Surjections}, we extend the map $\alpha$ to a map $L_M \to G$, which is no longer injective.
Let us examine its (failure of) injectivity.

%

The language accepted by $\Gamma_M$ consists of words mapping under $\alpha$ to the monoid $M \subset K$
generated by $H$ and $g^m$.
If $k=h_0g^mh_1g^m\cdots g^m h_r$ is an element of $M$ whose normal form in $K$ has $r$ appearances of $g^m$,
 then $k$ is hit by exactly $|F|^r$ elements of $L_M$, because $M \subset K \cong H *_F Z$. Since the weight of the path representing $g^m$ is $1/|F|$, each of these words in $L_M$ has weight $|F|^{-r}$, hence the total weight of $\alpha^{-1}(k) \subset L_M$ is $1$.
(Recall from \refsec{Perron} that the \emph{weight} of a word is the product of the weights of its letters.)

Since $\beta: G \to Gb$ has fibers of constant cardinality $C$, it follows from the above paragraph that for
 each $k\in M$, the total weight of of the words mapping to $kb$ is
 \[
 w(\alpha^{-1} \beta^{-1}(kb)) = C.
 \]

By \refthm{GroupoidLanguage}\refitm{SameLength}, a word of length $n$ in $L_M$ maps to a path of length $\leq n$, hence the endpoint of this path lies in the ball of radius $n$ about $b \in \Upsilon$. Therefore, letting $f_M(n)$ be the total weight of the words of length at most $n$ in $L_M$, we have
\begin{equation}\label{Eqn:KGrowthBis}
f_M(n) \leq C f_{G,\Upsilon}(n).
\end{equation}

Let $\rho_H$ and $\rho_M$ be the Perron--Frobenius eigenvalues of the adjacency matrices of $\Gamma_H$ and $\Gamma_M$, respectively.
Then
\[
\lambda_H(\Upsilon) =  \rho_H <  \rho_M = \lim_{n \to \infty} \sqrt[n]{ f_M(n)} \leq \lim_{n \to \infty} \sqrt[n]{C f_{G,\Upsilon}(n) } = \lambda_G (\Upsilon).
\]
Here, the first equality holds by \refeqn{GroupoidGrowth}. The strict inequality holds by \refthm{Domination}, since $\Gamma_H$ is a proper subgraph of the strongly connected graph $\Gamma_M$. The next equality holds by \refthm{PathCount}. The non-strict inequality holds by \refeqn{KGrowthBis}, and the final equality is by the definition \refeqn{LambdaDef} of $\lambda_G$.
\end{proof}

\section{Interlude: Alternate approaches to \refthm{GrowthGeneral}}\label{Sec:Alternate}

After the first version of this paper was distributed, several mathematicians informed us that \refthm{GrowthGeneral} can also be derived from various modern tools. 
 In this section, we survey two alternate approaches: one using growth tightness and a second using Patterson--Sullivan measures.

\subsection{Growth tightness and regular languages}\label{Sec:Tightness}

Let $G$ act properly and cocompactly on a graph $\Upsilon$. As in Construction~\ref{Const:FreeAction}, we may assume without loss of generality that $G$ acts without inversion. For any normal subgroup $N$, the quotient $G/N$ acts properly and cocompactly on the quotient graph $N \backslash \Upsilon$. We say that the action of $G$ on $\Upsilon$ is \emph{growth tight} if, for any infinite normal subgroup $N$,
\begin{equation}\label{Eqn:GrowthTight}
\lambda_{G/N}(N \backslash \Upsilon) < \lambda_G(\Upsilon).
\end{equation}

Grigorchuk and de La Harpe introduced growth tightness in the context
of Cayley graphs \cite{grigorchuk-delaharpe}, and proved that the
property holds for free groups with respect to free generating
sets. Arzhantseva and Lysenok 
showed that hyperbolic groups are growth tight with respect to any
generating set \cite{arzhantseva-lysenok}. Sambusettti proved growth
tightness for free products and several other classes of groups
\cite{sambusetti:free-products}. Yang \cite{Yang:GrowthTightness} studied groups with so-called
contracting elements, and in particular proved that non-elementary relatively
hyperbolic groups are growth tight.

 Arzhantseva, Cashen, and Tao generalized the definition to the context of group actions on metric spaces \cite{arzhantseva-cashen-tao}. Through this lens, they recovered all previously known examples of tightness, and extended the result to several new contexts (for instance, CAT(0) cube complexes). The above definition is a special case of theirs.

We now restrict to the case where $G$ is a hyperbolic group with a finite symmetric generating set $S$. Following \refthm{Automatic}, let $L = L_G$ be a short-lex geodesic language mapping bijectively to $G$.
For a constant $C >0$, we say that elements $x,y$ are \emph{$C$--close} if $x=gyh$ such that $|g|, |h| \leq C$. We say that \emph{$x$ $C$--contains $y$} if the short-lex geodesic word $\overline{x} \in L$ representing $x$ contains a subword that is $C$--close to $y$. Given $w \in G$, define
\[
X(w,C) = \{x \in G : \text{$x$ does not $C$--contain $w$} \}.
\]

 The following theorem of Arzhantseva and Lysenok is  the main technical result of \cite{arzhantseva-lysenok}. See \cite[Theorem 2]{arzhantseva-lysenok}. It implies growth tightness for hyperbolic groups \cite[Theorem 1]{arzhantseva-lysenok}, and also \refthm{GrowthGeneral} for the case where $\Upsilon$ is the Cayley graph of $G$.

 \begin{theorem}\label{Thm:StrongContainmentGrowth}
 Let $G$ be a non-elementary hyperbolic group with with finite generating set $S$ and Cayley graph $\Upsilon = \Upsilon(G,S)$. Then there is a constant $C = C(G,S)$ such that
 \[
 \lambda_{X(w,C)}(\Upsilon) < \lambda_G(\Upsilon).
 \]
 \end{theorem}

Given an infinite index quasiconvex subgroup $H \subset G$, and $C$ as in  \refthm{StrongContainmentGrowth}, choose a geodesic $\gamma$ in $\Upsilon$ from $1$ to $w$, such that $w$ is very far from $H$ (as a function of $C$ and the quasiconvexity constant). Then, for every $g \in G$, the translated geodesic $g \gamma$ must have at least one endpoint far from $H$. Thus, by the quasiconvexity of $H$, no geodesic path from $1$ to $h \in H$ can contain a subpath $C$--close to $w$. It follows that $H \subset X(w,C)$, hence \refthm{StrongContainmentGrowth} gives
 \[
\lambda_H(\Upsilon) \leq \lambda_{X(w,C)}(\Upsilon) < \lambda_G(\Upsilon),
 \]
establishing \refthm{GrowthGeneral} for Cayley graphs.

\subsection{Patterson--Sullivan measures}\label{Sec:Measure}

For a group $G$ with generating set $S$, the \emph{Poincar\'e series} is
\[
\zeta_G(x) = \sum_{g \in G} e^{- x |g|},
\]
where $|g|$ denotes the length of $g$ in the generating set. If the growth rate $\lambda_G$ is a well-defined limit, there is a \emph{critical exponent} $h(G,S) = \log \lambda_G (\Upsilon(G,S))$ 
 such that $\zeta_G(x)$ converges for all $x < h(G,S)$ and diverges for $x > h(G,S)$. The Poincar\'e series is used to construct a probability measure on $G \cup \bdy G$, called the Patterson--Sullivan measure. See Coornaert \cite{Coornaert93} or Calegari \cite[Section 2.5]{calegari:ergodic-groups} for more detail.

The group $G$ is said to be \emph{of divergence type} if $\zeta_G(x)$ diverges at $x = h(G,S)$. One immediate consequence of Cannon's work on the growth of regular languages (more precisely, of \refprop{PolyExp} and \refcor{CannonGrowth}) is that hyperbolic groups have divergence type. We can now state the following theorem of Matsuzaki, Yabuki, and Jaerisch \cite[Corollary 2.8]{MYJ}, restated in the notation of our paper.

\begin{theorem}\label{Thm:MYJ}
Let $G$ be a non-elementary group acting discretely on a hyperbolic metric space $\Upsilon$. Suppose that
$H \subset G$ is a subgroup of divergence type, and that the limit set $\Lambda(H)$  is a proper subset of  $ \Lambda(G)$. Then $\lambda_H(\Upsilon) < \lambda_G(\Upsilon)$.
\end{theorem}

We observe that \refthm{GrowthGeneral}
follows as a consequence of \refthm{MYJ}, because quasiconvex subgroups of hyperbolic groups have divergence type, and quasiconvex subgroups of infinite index have limit set $\Lambda(H) \subsetneq \Lambda(G)$.

\section{Beyond hyperbolic groups}\label{Sec:BeyondHyperbolic}

The goal of this section is to prove Theorems~\ref{Thm:RelHyperbolic} and~\ref{Thm:Cubulated}, which were stated in the introduction. Although the hypotheses of these theorems are quite different (one  concerns relatively hyperbolic groups, the other cubulated groups), the proof strategy is the same. Both proofs rely on the notion of 
growth tightness, defined in \refsec{Tightness}, as well as weak proper discontinuity, which we define in \refsec{WPD}.

In \refprop{SmallGrowthTight}, we observe that when $G$ admits both a growth tight action on $\Upsilon$ and a weakly properly discontinuous action on a hyperbolic space $X$, the group $G$ grows faster than any subgroup acting elliptically on $X$. This result may be of independent interest.

In \refsec{RelHyperbolic}, we define relative hyperbolicity and present the vocabulary and tools related to cone-off
constructions. Then we construct the space $X$ required to apply \refprop{SmallGrowthTight} and prove \refthm{RelHyperbolic}.

In \refsec{Cubulated}, we recall the definitions of CAT(0) cube complexes and their hyperplanes. Then we prove  \refthm{Cubulated}, again relying on \refprop{SmallGrowthTight}.

\subsection{Weak proper discontinuity meets tightness}\label{Sec:WPD}

Suppose that $G$ acts by isometries on a hyperbolic metric space $X$. We say that the action is \emph{weakly properly discontinuous (WPD)} if the following conditions hold:
\begin{enumerate}
\item $G$ is non-elementary,
\item $G$ contains at least one loxodromic element on $X$, and
\item for every loxodromic $g \in G$, and every $x \in X$, and every $r > 0$, there exists $n > 0$ such that the set
\(
\{ \gamma \in G  : \dist(x, \gamma x) \leq r, \: \dist (g^n x, \gamma g^n x ) \leq r \}
\)
is finite.
\end{enumerate}

This notion was introduced by Bestvina and Fujiwara \cite{Bestvina-Fujiwara:bounded-cohomology}. The main observation of this section is that growth tightness combined with the WPD property implies that $G$ grows faster than any subgroup acting elliptically. We say that the action of $H \subset G$ on $X$ is \emph{elliptic} if there is a bounded orbit $Hx \subset X$.

\begin{proposition}\label{Prop:SmallGrowthTight}
Let $G$ be a group with a proper, cocompact, and growth tight action on a graph $\Upsilon$.
Assume also that $G$ has a WPD action on a hyperbolic space $X$.
Suppose that $H \subset G$ is a subgroup whose action on $X$ is elliptic. 

Then the
growth rate of $H$ for $\Upsilon$  is strictly smaller
than the growth rate of $G$:
$$  \lambda_H( \Upsilon)   <  \lambda_G(\Upsilon). $$
\end{proposition}


\begin{proof}
Let $\gamma$ be a loxodromic element of $G$ on $X$. By
\cite[Theorem 8.7]{DGO}, there exists $m>0$ such that in the normal closure
$N = \nclose{\gamma^m}_G$,
all non-trivial
elements are loxodromic on $X$. In particular, this subgroup has trivial intersection with $H$. Thus the
quotient map $G \to \overline{G} =  G/N$ restricts to an embedding $H \to \overline{H} \subset \overline{G}$. The quotient map $\Upsilon \to \overline{\Upsilon} =  N \backslash \Upsilon$ is also
$1$--Lipschitz, implying that   $ f_{ H,
  \Upsilon} (n) \leq f_{\overline{H},  \overline{\Upsilon} } (n)$ for every $n$. Thus the growth rates satisfy
 \[
  \lambda_{H}(\Upsilon)  \leq
    \lambda_{\overline{H} }(\overline{\Upsilon} )  \leq
   \lambda_{\overline{G} }(\overline{\Upsilon} ) <
 \lambda_{
  G}(\Upsilon).
\]
 Here, the first inequality follows from the above inequality on $f_H$, the 
 second inequality is by set containment, and the final strict inequality
follows from \refeqn{GrowthTight} because we have assumed the $G$--action on $\Upsilon$ is
growth tight.
 \end{proof}

\subsection{Relatively hyperbolic groups}\label{Sec:RelHyperbolic}

In order to apply \refprop{SmallGrowthTight} to relatively hyperbolic groups, we recall the definitions.

Consider a 
group $G$ with a finite generating set $S$.
A \emph{peripheral structure} 
$\mathcal{P}$ is a collection of subgroups closed under conjugation. 
We will work with peripheral structures containing finitely many conjugacy classes. Let 
$P_1,\dots,
P_k \in \mathcal{P}$ be subgroups representing the conjugacy classes in $\mathcal{P}$.

The \emph{coned-off Cayley graph} for $S$ over $\{P_1,\dots,
P_k\}$, denoted $X_0$,  is obtained from the Cayley graph $\Upsilon(G,S)$ by adding
a vertex for each left coset of each $P_i$ and linking it by length
$1$ edges to every element of the coset. Observe that $G$ acts by isometries on $X_0$.

The \emph{angular distance} at a vertex $v$ in
$X_0$ is the distance $\hat{d}:{\rm link}(v) \times {\rm link}(v)\to
\mathbb{N}\cup \{\infty\}$ defined by
the length of a shortest path in $X_0$ between $a, b\in {\rm link}(v)$ that
does not contain $v$.

We say that $(G,\mathcal{P})$ is \emph{relatively hyperbolic} if $X_0$ is
Gromov hyperbolic and if the angular distance at each vertex is
locally finite. See e.g.\ Hruska \cite{HruskaRelQC} for other equivalent definitions.

Observe that $X_0$ naturally contains $\Upsilon(G,S)$. Thus we say that
a subgroup $H \subset G$ is \emph{relatively quasiconvex} in $(G,\mathcal{P})$ if  its
image in $X_0$ is quasiconvex \cite{HruskaWisePacking}. 
The properties of relative hyperbolicity and relative quasiconvexity are invariant under 
quasi-isometry, hence do not depend on the generating set $S$.

\begin{proposition}\label{Prop:SmallConstruction}
Let $(G,\mathcal{P})$ be a non-elementary relatively hyperbolic group. Suppose that $H$ is
an infinite-index subgroup of $G$, which is
relatively quasiconvex  in $(G,\mathcal{P})$. Then $G$ has an isometric, WPD action on a hyperbolic
space $X_h$. Furthermore, the action of $H$
is elliptic.
\end{proposition}

In fact, we will prove the stronger statement that the action of $G$ on $X_h$ is \emph{acylindrical}.
 This means that for all $r > 0$ there exist $R, N>0$ such that for all $x,y\in X$ with
   $\dist(x,y)\geq R$, the set $\{\gamma \in G : \dist (x, \gamma x) \leq r, \: \dist(y, \gamma y)  \leq r\}$
    has cardinality at most $N$. This property is stronger than WPD because $N$ is uniform over $G$ and $y$ is less restricted than the element $g^n x$ in the definition of WPD.

\begin{proof}[Proof of \refprop{SmallConstruction}]
Let $X_0$ denote the coned-off Cayley graph of $G$, as above.
Hruska and Wise \cite[Theorem 1.4]{HruskaWisePacking} proved that $H$ has \emph{finite
relative height}:
 there exists a smallest integer $h \geq 0$ 
 such that any intersection of essentially distinct
conjugates $g_iHg_i^{-1}$ for $i=0, \dots, h$ has finite diameter in $X_0$. 
(Here, one says that
conjugates $g_iHg_i^{-1}$  are \emph{essentially distinct} if the
cosets $g_iH$ are all distinct.)

Observe that a subgroup of a peripheral group $P_i$ has relative height $0$, and a malnormal subgroup $H \subset G$ has relative height $1$.

For $1 \leq i \leq h$, let $\mathcal{H}_i$ denote  the collection of intersections of
$i$--tuples of essentially distinct conjugates of $H$. In this way,
$\mathcal{H}_{h+1}$ consists of finite or parabolic groups, but, if
$h\neq 0$, $\mathcal{H}_{h}$ contains groups with infinite diameter in
$X_0$.  We may choose
conjugacy representatives $(\mathcal{H}_i)_0$ for elements in $\mathcal{H}_i$, and
 let 
$\mathcal{C}\mathcal{S}\mathcal{H}_i$ denote the collection of the
left cosets of the groups of $(\mathcal{H}_i)_0$ that are finite or
parabolic, and of the left cosets of the stabilizers of the limit sets
in the boundary 
of the other groups of $(\mathcal{H}_i)_0$.

Starting from $X_0$, let $X_{i}$ be the cone-off of $X_{i-1}$ over the collection of subsets  $\mathcal{C}\mathcal{S}\mathcal{H}_{h+1-i}$.

In \cite[Corrigendum Theorem 3]{Dahmani-Mj}, \begin{com}{added reference and the term \emph{saturated}}\end{com}  Dahmani and Mj proved that $G$, with the relative metric of
$X_0$, has saturated  graded relative hyperbolicity with respect to $H$. This means that every
$X_i$ is hyperbolic  for $i =0, \ldots, h$, and
the  angular distance at each cone-vertex is bounded from below by a
proper function of the distance in $X_{i-1}$.  
The key fact in our setting is that $X_i$ is hyperbolic for all $i$
and that   the
elements of  $\mathcal{C}\mathcal{S}\mathcal{H}_{h-i+1}$ are uniformly
quasiconvex, and mutually cobounded  in the metric of $X_{i-1}$. See \cite[Corrigendum Proposition 2]{Dahmani-Mj}. Here, \emph{mutually cobounded} means that for
any pair of distinct cosets, the shortest point projection of one
to the other has uniformly bounded diameter.

By the time the inductive construction reaches $X_{h}$, the collection $\mathcal{C}\mathcal{S}\mathcal{H}_{1}$ has been
 coned off, which means the cosets of $H$ have been coned-off.

\begin{lemma}\label{Lem:Acylindrical}
$X_h$ is hyperbolic, and the action of $G$ on $X_{h}$ is acylindrical.
\end{lemma}

\begin{proof} We will use \cite[Proposition 5.40]{DGO}, which ensures that if $X$ is a
hyperbolic graph, with an isometric $G$--action,  and if
$\mathcal{Q}$ is an invariant   collection of uniformly
quasiconvex subspaces that are mutually cobounded, then the action
of $G$ on  a certain
cone-off  of $X$ over $\mathcal{Q}$  is acylindrical.  We take the
precaution of saying ``a certain cone-off'' because the construction of
\cite{DGO} is different from the the one we defined here (in particular, the radius
of the cones is much larger). However, there is
an equivariant quasi-isometry between both cone-offs, and acylindricity is
preserved by equivariant quasi-isometries. Thus,  the cone-off construction we
use here, on a hyperbolic space, over
a collection of uniformly quasiconvex, uniformly mutually cobounded subsets, preserves
the acylindricity of the action.

From here, the proof proceeds by induction on height.
For the base case, note that the action of $G$ on $X_0$ is
acylindrical by the work of Osin \cite[Proposition 5.2]{Osin:AcylindricallyHyperbolic}. This also follows
from the above \cite[Proposition 5.40]{DGO}  applied to the coning-off of horoballs in a
cusp-uniform space associated to $(G,\mathcal{P})$, which gives a space
equivariantly quasi-isometric to $X_0$ (see \cite[Proposition 2.8]{Dahmani-Mj}).

 Assume
that the action of $G$ is acylindrical on $X_i$.
 The construction of $X_{i+1}$
is a coning-off of a family of quasiconvex, mutually cobounded subsets of the hyperbolic
space  $X_i$.
By \cite[Proposition 5.40]{DGO}, combined with the equivariant quasi-isometry described above, the action of $G$ is therefore  acylindrical
on $X_{i+1}$.

Thus by induction, $X_{h}$ is hyperbolic and the action of $G$ on $X_{h}$ is
acylindrical.
\end{proof}

Notice that the action of $H$ on $X_{h}$ is elliptic since it
has been coned-off. Furthermore, since $H$ has infinite index in $G$
and is relatively quasi-convex in $(G, \mathcal{P})$ which is non-elementary relatively hyperbolic,
its limit set is not the whole boundary of $G$,  and the stabilizer of its
limit set  has infinite index in $G$. 
The
diameter of $X_{h}$ is infinite.   Thus, by a theorem of Osin \cite[Theorem 1.1]{Osin:AcylindricallyHyperbolic}, the action of $G$ on $X_h$ contains loxodromic elements. 
\end{proof}

We can now restate and prove \refthm{RelHyperbolic}.

\begin{named}{\refthm{RelHyperbolic}}
Let $(G,\mathcal{P})$ be a non-elementary relatively hyperbolic group, and $H$ a relatively
quasiconvex subgroup of $(G,\mathcal{P})$ of infinite index in $G$. Suppose that $G$ acts properly and cocompactly on a graph $\Upsilon$.
Then
\[
\lambda_H(\Upsilon) < \lambda_G (\Upsilon).
\]
\end{named}

\begin{proof}
We need to check the  the hypotheses of 
\refprop{SmallGrowthTight}. Indeed,  \refprop{SmallConstruction} provides a WPD action by $G$ on a hyperbolic space $X = X_h$, where $H$ acts elliptically. The growth tightness of non-elementary relatively hyperbolic
groups is a theorem of Yang  \cite[Corollary
1.7]{Yang:GrowthTightness} and Arzhantseva--Cashen--Tao \cite[Theorem 8.6]{arzhantseva-cashen-tao}. Thus, by \refprop{SmallGrowthTight}, we have $\lambda_H(\Upsilon) < \lambda_G (\Upsilon)$.
\end{proof}

\subsection{Cubulated groups}\label{Sec:Cubulated}

\refprop{SmallGrowthTight} also enables us to extend \refthm{GrowthGeneral} to the context of groups acting on CAT(0) cube complexes. We recall the definitions very quickly, while pointing the reader to e.g.\ \cite{HagenArboreal}  or \cite{WiseCBMS2012} for a detailed treatment.

For $0 \leq n < \infty$, an \emph{$n$-cube} is $[-\frac{1}{2}, \frac{1}{2}]^n$. A \emph{cube complex} is the union of a number of cubes, possibly of different dimensions, glued by isometry along their faces. A cube complex $\calX$ is called  \emph{CAT(0)} if it is simply connected, and if the link of every vertex is a flag simplicial complex. 

A CAT(0) cube complex $\calX$ has two natural metrics: the $L^1$ or \emph{combinatorial metric} that metric that agrees on $\calX^{(0)}$ with the graph metric on $\Upsilon = \calX^{(1)}$; and the $L^2$ or \emph{CAT(0) metric} obtained by extending the Euclidean path-metric on the cubes. (By a theorem of Gromov, the $L^2$ metric indeed satisfies the CAT(0) inequality for geodesic triangles. See Leary \cite{Leary_KanThurston} and the references therein.) The theorem below is valid in either metric on $\calX$.

A \emph{midcube} of an $n$-cube $C$ is an $(n-1)$ cube obtained by restricting one coordinate of $C$ to $0$. A \emph{hyperplane} $V \subset \calX$ is a connected union of midcubes, with the property that $V$ intersects every cube of $\calX$ in a midcube or in the empty set. Every hyperplane separates $\calX$. A hyperplane is  $V$ called \emph{essential} if both components of $\calX - V$ contain points arbitrarily far from $V$.

The \emph{carrier} of a hyperplane $V$ is the subcomplex of $\calX$ consisting of all cubes that meet $V$.
The adjacency of hyperplane carriers can be encoded in the \emph{contact graph} $\calC \calX$, introduced by Hagen \cite{HagenArboreal}. The vertices of this graph are hyperplanes of $\calX$, and hyperplanes $V,W$ are connected by an edge of $\calC \calX$ if and only if their carriers are disjoint. Hagen proved that $\calC \calX$ is a quasi-tree, and in particular is hyperbolic \cite{HagenArboreal}.

With this background, we can now state and prove the following result.

\begin{named}{\refthm{Cubulated}}
Let $G$ be a non-elementary group, acting properly and cocompactly on a CAT(0) cube complex $\calX$. Suppose that $\calX$ does not decompose as a product. Then, for every subgroup $H \subset G$ stabilizing an essential hyperplane of $\calX$, we have
\[
\lambda_H(\calX) < \lambda_G (\calX).
\]
\end{named}

\begin{proof}
First, we may assume without loss of generality that all hyperplanes of $\calX$ are essential. Otherwise, replace $\calX$ by its essential core $\mathcal{Y}$, as provided by the essential core theorem of Caprace and Sageev \cite[Proposition 3.5]{CapraceSageev2011}. The hyperplanes of $\mathcal{Y}$ will be in bijective correspondence with the essential hyperplanes of $\calX$. Furthermore, $\calX$ is contained in a bounded neighborhood of $\calY$,  hence $\lambda_G(\calX) = \lambda_G(\calY)$ and $\lambda_H(\calX) = \lambda_H(\calY)$.

Under the hypotheses of the theorem, Caprace and Sageev proved that some element $g \in G$ has a rank-one action on $\calX$. This means that $g$ acts by translation on a combinatorial geodesic axis $A$, and furthermore this axis does not bound a half-plane \cite{CapraceSageev2011}. Since all hyperplanes of $\calX$ are essential, the construction of \cite[Section 6.1]{CapraceSageev2011} produces an element $g$ such that no power $g^n$ stabilizes a hyperplane. Furthermore, there is a hyperplane $V$ intersecting $A$, such that $V$ and $gV$ are not neighbors in $\calC \calX$.

By a theorem of Behrstock, Hagen, and Sisto \cite[Theorem A]{BHS:curve-complex-cubical}, the $G$--action on $\calC \calX$ is WPD. In fact, the rank-one element $g$ produced by Caprace and Sageev is loxodromic on $\calC \calX$. Meanwhile, by the definition of $\calC \calX$, any subgroup $H \subset G$ stabilizing a hyperplane $V \subset \calX$ necessarily fixes a vertex of $\calC \calX$.

We claim that the isometry $g$ is \emph{contracting}: this means that  that every ball in $\calX$ disjoint from the axis $A$ has universally bounded projection to $A$. For the CAT(0) metric on $\calX$, this is a theorem of Bestvina and Fujiwara \cite[Theorem 5.4]{Bestvina-Fujiwara:characterization}. For the combinatorial metric on $\calX$, this is  a theorem of Genevois \cite[Theorem 3.10]{Genevois16}. (Alternately, one may use 
an argument of Huang \cite{Huang:Contracting} to transfer the Bestvina--Fujiwara conclusion to the cubical metric.) As a consequence of the claim, 
a theorem of Arzhantseva, Cashen, and Tao \cite[Theorem 6.4]{arzhantseva-cashen-tao} and Yang \cite[Theorem 1.3]{Yang:GrowthTightness} says that the $G$--action on $\calX$ is growth tight. 

We have now checked all the hypotheses of \refprop{SmallGrowthTight}, with $\Upsilon = \calX^{(1)}$ and $\calC \calX$ playing the role of $X$. Thus we have an inequality of growth rates:
\[ \lambda_H(\calX) < \lambda_G(\calX). \qedhere
\]
\end{proof}

\section{Examples and open problems}\label{Sec:Examples}

This section explores the extent to which the hypotheses of Theorems~\ref{Thm:GrowthGeneral}--\ref{Thm:Cubulated} can be loosened, or the conclusions strengthened.

\subsection{Uniform bounds on growth}
Given \refthm{GrowthGeneral}, one may ask whether there exists a uniform upper bound $\alpha < \lambda_G$
such that $\lambda_H \leq \alpha$ for each infinite index quasiconvex subgroup $H \subset G$. For instance, Corlette showed that this is the case with lattices in quaternionic hyperbolic spaces \cite{Corlette90}. 

\begin{example}
For $k \geq 2$, let $H_{\HH}^k$ be $k$--dimensional quaternionic hyperbolic space. Every cocompact lattice $G \subset \mathrm{Isom} (H_{\HH}^k)$ is a hyperbolic group. Let $\Gamma \subset G$ be an infinite-index subgroup. Corlette \cite{Corlette90} showed that the growth rates of $G$ and $\Gamma$ with respect to the action on  $H_{\HH}^k$ satisfy
\[ \lambda_{\Gamma}(H_{\HH}^k) \leq e^{4k} < e^{4k+2} = \lambda_{G}(H_{\HH}^k).
\]
\end{example}

On the other hand, in \refthm{NonUniform} below, we show that no such gap between $G$ and its subgroups can exist when $G = F_2$.

Before giving the construction, we recall some graph terminology. The \emph{girth} of a graph is the length of the shortest cycle. If $T$ is a tree, i.e.\ a graph with no cycles, the girth of $T$ is infinite. A \emph{rooted tree} is a tree $T$, with a fixed vertex designated as the root. Every rooted tree can be directed outward from the root.
A \emph{leaf} in a directed tree is a vertex with no outgoing edges.

For an integer $k \geq 1$, a \emph{$k$--tree} is a rooted tree where each leaf is reachable from the root by a geodesic of length $k$,
and every other vertex has 3 outgoing edges, except for at most one vertex having 2 outgoing edges.

\begin{lemma}\label{Lem:kTree}
 Let $T_k$ be a $k$--tree.
Then the number of leaves in $T_k$ is at least $2\times 3^{k-1}$, with the lower bound realized when the root has 2 outgoing edges.
\end{lemma}
\begin{proof}
This holds by induction on the distance from a 2--vertex to a leaf.
\end{proof}

Let $T$ be a directed tree. For a vertex $v \in T$ and for $k \geq 1$, define \emph{depth~$k$ subtree} $T_k(v) \subset T$ to be the subtree reachable by directed paths of length $\leq k$ from $v$. We think of $v$ as the root of $T_k(v)$.

\begin{lemma}\label{Lem:TreeGrowth}
 Let $T$ be a rooted tree with root $b$. Suppose that for every $v \in T$, the depth~$k$ subtree $T_k(v)$  is a $k$--tree. Then there is a constant $C > 0$ such that for $n \in \naturals$,
\[
f_T(n) \: = \:  \# \left\{ v\in T: \dist_T(b, \, v) \leq n \right\}  \: \geq \: C \left( 3 \times \left( \tfrac{2}{3} \right)^{1/k} \right)^n.
\]
\end{lemma}

\begin{proof}
The growth function $f_T(n)$ is bounded below by the spherical growth function
\[
s_T(n) = \# \left\{ v\in T: \dist_T(b, \, v) = n \right\}.
\]
By \reflem{kTree}, we have
\[
s_T(n+k) \geq \left( 2 \times 3^{k-1} \right) s_T(n),
\]
hence the result follows by induction. The first $k$ values of $s_T$ form the base case of the induction, and  determine the constant $C$.
\end{proof}

\begin{theorem}\label{Thm:NonUniform}
Let $G = \langle g_1, g_2 \rangle$. Let $\Upsilon$ be the Cayley graph of $G$ with respect to the free generators. Then there is a sequence of infinite index quasiconvex subgroups $H_k \subset G$, such that
\[
\lim_{k \to \infty} \lambda_{H_k}(\Upsilon) =  3 = \lambda_G(\Upsilon).
\]
\end{theorem}

\begin{proof}
 Let $B$ be a bouquet of two circles. Then we may identify  $\Upsilon$ with $\widetilde B$, so that the basepoint $b \in \widetilde B$ corresponds to $1 \in G$. The spherical growth rate of $G$ is
 $s_{G,\Upsilon}(n) =  4 \times 3^{n-1}$. Thus
 \[
f_{G,\Upsilon}(n) = \sum_{i=0}^{n} s_{G,\Upsilon}(i) = 2 \times 3^n -1
\qquad \text{and} \qquad
\lambda_G(\Upsilon) = \lim_{n \to \infty} \sqrt[n]{ f_{G,\Upsilon}(n) } =  3.
\]

By residual finiteness of $G = \pi_1 B$, for each $k \geq 1$ let
$B_k$ be a finite based cover whose girth is at least $2k+1$.
Let $A_k$ be obtained from  $B_k$ by removing a single edge at the basepoint. The endpoints of the removed edge become trivalent in $A_k$.
Then the shortest geodesic path in $A_k$ starting and ending at a trivalent vertex has length at least $2k$, because
 $\girth (B_k)  \geq 2k+1$. Let $H_k = \pi_1(A_k)$.

Consider the based universal cover $\widetilde A_k$ as a subtree of $\Upsilon = \widetilde B$. The root
 $b \in \widetilde A_k$ is a trivalent vertex corresponding to $1 \in \Upsilon$. The vertices of $\widetilde A_k$ that
 lie in the $H_k$--orbit of $b$
form an equidistributed subset of density $1/[G: \pi_1 B_k]$.

Note that $\widetilde A_k$ satisfies the hypotheses of \reflem{TreeGrowth}. This is because every vertex of $\widetilde A_k$ with two outgoing edges is trivalent, and every pair of trivalent vertices in $\widetilde A_k$ are distance at least $2k$ apart. Thus there is a positive constant $C_1$ such that
\[
f_{H_k,\Upsilon}(n) \geq \frac{1}{[G: \pi_1 B_k]} \, f_{\widetilde A_k}(n) \geq C_1 \left( 3 \times \left( \tfrac{2}{3} \right)^{1/k} \right)^n .
\]
Here, the first inequality comes from the density of the $H_k$--orbit of $b$ in $\widetilde A_k$, and the second inequality is by \reflem{TreeGrowth}.
Therefore,
\[
3 \geq \lambda_{H_k}(\Upsilon) = \lim_{n \to \infty} \sqrt[n]{ f_{H_k,\Upsilon}(n)} \geq   \left( 3 \times \left( \tfrac{2}{3} \right)^{1/k} \right),
\]
hence $\lambda_{H_k}(\Upsilon)$ converges to $ 3$ as $k \to\infty$.
\end{proof}

The proof of \refthm{NonUniform} is elementary, and needs none of the tools used in the earlier sections, as counting vertices is easier  in trees than in general Cayley graphs.
\refthm{NonUniform} has been generalized  in \cite{LiWise2017}
to assert that if $X$ is a compact special cube complex, then there is a sequence of infinite index quasiconvex subgroups of $\pi_1X$ whose growth rates converge to the growth rate of $\pi_1X$.

\subsection{The need for quasiconvexity}
The following examples show that some version of the quasiconvexity hypothesis is crucial for \refthm{GrowthGeneral}.

\begin{example}\label{Ex:Fiber}
Let  $G = \pi_1(M)$, where $M$ is a closed hyperbolic $3$--manifold that fibers over the circle. This fibration induces a short exact sequence
\[
1 \to H \to G \to \ZZ \to 1.
\]
It is well known that the fiber subgroup $H$ is highly distorted in $G$. See e.g.\ \cite{cannon-thurston}.

Since $G/H \cong \ZZ$, for any Cayley graph $\Upsilon$ the growth function $f_{G,\Upsilon}(n)$ is at most linearly larger than $f_{H,\Upsilon}(n)$, hence $\lambda_H(\Upsilon) = \lambda_G(\Upsilon)$.
\end{example}

\begin{example}\label{Ex:Rips}
Rips \cite{Rips82} observed that for each finitely presented group $Q$ there exists a short exact sequence
\[
1\rightarrow N \rightarrow G \rightarrow Q \rightarrow 1,
\]
where $G$ is a $C'(\frac16)$ small-cancellation group (hence hyperbolic),
and where $N$ is generated by two elements.
This provides normal subgroups $N\subset G$ with exotic growth properties. In particular, when $Q$ has sub-exponential growth (e.g. $Q \cong \ZZ$, as in  \refex{Fiber}), it follows that
$\lambda_N = \lambda_G$.
\end{example}

\subsection{Open questions}\label{Sec:Problems}

\begin{prob}\label{Prob:Cubulated}
Generalize \refthm{Cubulated} to cubically convex subgroups of a cubulated group $G$. Do all such subgroups grow slower than $G$ itself? Given \cite[Theorem 9.2]{arzhantseva-cashen-tao} and  \refprop{SmallGrowthTight}, the challenge is to find some hyperbolic space on which $G$ is WPD but $H \subset G$ acts elliptically. It may be possible to obtain such a space by coning off certain $G$--orbits in the contact graph $\calC \calX$, as in \refsec{RelHyperbolic}.

One could also approach this problem using regular languages.
By the work of Niblo and Reeves \cite{NibloReeves98}, cubulated groups are biautomatic. Thus \refthm{PathCount} and the rest of Perron--Frobenius theory are already available for this problem. The challenge is to construct a free product as in \refthm{FreeProduct} or \refthm{FreeProductAlternate}, or else to circumvent this construction.

\end{prob}

\begin{prob}
Find an analogue of \refthm{GrowthGeneral} that works in the general setting of automatic groups, without any geometric hypotheses. That is, let $L_G$ be an automatic structure for $G$, with some generating set $S$. 
Is there a language--theoretic description of the automatic subgroups $H \subset G$ such that $f_{L_H}$ grows exponentially more slowly than $f_{L_G}$?  This question is closely related to the problem, studied by Ceccherini-Silberstein and Woess  \cite{CSW1, CSW2}, of deciding what languages are \emph{growth sensitive}, that is, what languages have the property that prohibiting a set of sub-words reduces the growth rate of the language.
\end{prob}

\begin{prob}
Suppose $G$ is a finitely generated group, and $H$ is a non-trivial subgroup such that there is a monoid embedding $H*\naturals \subset G$.
Show that the growth rate of $G$ is larger than that of $H$. If the subgroup $H$ has divergence type, one can likely do this using Poincar\'e series, as in \cite{sambusetti:free-products} and \cite[Section 6]{arzhantseva-cashen-tao}.
\end{prob}


\bibliographystyle{alpha}
\bibliography{wise}

\end{document}